\documentclass[12pt]{amsart}
\usepackage{amscd,amsmath,amsfonts,amssymb,latexsym}
\usepackage[margin=3.3cm]{geometry}
\usepackage{hyperref,graphicx}
\usepackage{xcolor}
\hypersetup{
    colorlinks=true,
    citecolor=blue,
    linkcolor=blue,
    urlcolor=blue,
}
\usepackage{tikz}
\usepackage[all]{xy}
\usepackage{slashed}
\usepackage{soul}
\usepackage{comment}
\usepackage{extarrows}
\usepackage{mathtools}
\usepackage{url}
\usepackage{lipsum}

\numberwithin{equation}{section}
\theoremstyle{plain}
\newtheorem{lemma}{Lemma}[section]
\newtheorem{proposition}[lemma]{Proposition}

\newtheorem{theorem}[lemma]{Theorem}

\theoremstyle{definition}
\newtheorem{definition}[lemma]{Definition}
\newtheorem{remark}[lemma]{Remark}

\let\C\relax
\newcommand{\C}{{\mathbb C}}
\newcommand{\R}{{\mathbb R}}
\newcommand{\Z}{{\mathbb Z}}

\newcommand{\id}{{\rm id}}
\newcommand{\Om}{{\Omega}}
\newcommand{\om}{{\omega}}

\newcommand{\la}{\langle}
\newcommand{\ra}{\rangle}
\newcommand{\p}{{\partial}}

\newcommand{\vol}{\mbox{\rm vol}}
\newcommand{\rmspan}{\mbox{\rm span}}

\newcommand{\Sp}{{\text{\rm Spin}(7)}}

\renewcommand{\Im}{{ \rm Im \,}}

\newcommand{\Aa}{{\mathcal A}}

\newcommand{\Ff}{{\mathcal F}}

\renewcommand{\i}{{\sqrt{-1}}}
\renewcommand{\l}{{\ell}}

\newcommand{\n}{{\nabla}}
\renewcommand{\k}{{\kappa}}

\begin{document}

%\date{\today}

\title{The real Fourier--Mukai transform of Cayley cycles}

\author{Kotaro Kawai}
\address{Department of Mathematics, Faculty of Science, Gakushuin University, 1-5-1 Mejiro, Toshima-ku, Tokyo 171-8588, Japan}
\email{kkawai@math.gakushuin.ac.jp}

\author{Hikaru Yamamoto}
\address{Department of Mathematics, Faculty of Pure and Applied Science, University of Tsukuba, 1-1-1 Tennodai, Tsukuba, Ibaraki 305-8577, Japan}
\email{hyamamoto@math.tsukuba.ac.jp}

\thanks{The first named author is supported by 
JSPS KAKENHI Grant Number JP17K14181, 
and the second named author is supported by JSPS KAKENHI Grant Number 
JP18K13415 and Osaka City University Advanced Mathematical Institute (MEXT Joint Usage/Research Center on Mathematics and Theoretical Physics)}
%%%%%%%%%%%%%%%%%%%%%%%%%%%%%%%%%%%%%%%%%%%%%%%%%%%%%%%%%
\begin{abstract}
The real Fourier--Mukai transform sends 
a section of a torus fibration to a connection over the total space of the dual torus fibration. 
By this method, 
Leung, Yau and Zaslow introduced 
deformed Hermitian Yang--Mills (dHYM) connections for K\"ahler manifolds 
and Lee and Leung introduced 
deformed Donaldson--Thomas (dDT) connections for $G_2$- and ${\rm Spin}(7)$-manifolds. 

In this paper, 
we suggest an alternative definition of a dDT connection for a manifold 
with a ${\rm Spin}(7)$-structure which seems to be more appropriate 
by carefully computing the real Fourier--Mukai transform again. 
We also post some evidences showing that the definition we suggest 
is compatible with 
dDT connections for a $G_2$-manifold  
and 
dHYM connections of a Calabi-Yau 4-manifold. 

Another importance of this paper is that it motivates our study in our other papers. 
That is, based on the computations in this paper, 
we develop the theories of deformations of dDT connections for a manifold 
with a ${\rm Spin}(7)$-structure and the  ``mirror" of the volume functional, 
which is called the Dirac-Born-Infeld (DBI) action in physics. 
\end{abstract}
%%%%%%%%%%%%%%%%%%%%%%%%%%%%%%%%%%%%%%%%%%%%%%%%%%%%%%%%%
\keywords{mirror symmetry, 
deformed Donaldson--Thomas, 
special holonomy, calibrated submanifold}
\subjclass[2010]{
Primary: 53C07, Secondary: 53D37, 53C25, 53C38}
% [Memo]
% 53C07  	Special connections and metrics on vector bundles
% 58D27  	Moduli problems for differential geometric structures
% 58H15  	Deformations of structures
% 53D37  	Mirror symmetry, symplectic aspects; homological mirror symmetry; Fukaya category
% 53C25  	Special Riemannian manifolds
% 53C38  	Calibrations and calibrated geometries
\maketitle

\tableofcontents
%%%%%%%%%%%%%%%%%%%%%%%%%%%%%%%%%%%%%%%%%%%%%%%%%%%%%%%%%
%%%%%%%%%%%%%%%%%%%%%%%%%%%%%%%%%%%%%%%%%%%%%%%%%%%%%%%%%
%%%%%%%%%%%%%%%%%%%%%%%%%%%%%%%%%%%%%%%%%%%%%%%%%%%%%%%%%
%%%%%%%%%%%%%%%%%%%%%%%%%%%%%%%%%%%%%%%%%%%%%%%%%%%%%%%%%
%%%%%%%%%%%%%%%%%%%%%%%%%%%%%%%%%%%%%%%%%%%%%%%%%%%%%%%%%
\section{Introduction}
In the context of mirror symmetry, in particular Konstevich's homological mirror symmetry conjecture, one of vital needs is to provide a geometric functor from one side to its mirror side. Originally, the conjecture was stated for Calabi--Yau manifolds, however, the applicable scope has been extended to the other special holonomy cases, $G_2$ and ${\rm Spin}(7)$. Firstly, for the Calabi--Yau case, Leung, Yau and Zaslow \cite{LYZ} in 2000 found a natural and promising candidate for such a functor, 
which is called the real Fourier--Mukai transform nowadays. 

The real Fourier--Mukai transform sends 
a section of a torus fibration to a connection over the total space of the dual torus fibration. 
In their paper \cite{LYZ}, Leung, Yau and Zaslow proved that the real Fourier--Mukai transform of a special Lagrangian cycle 
is a \emph{deformed Hermitian Yang--Mills (dHYM) connection}. 
This can be considered as a correspondence between supersymmetric A-cycles and B-cycles in the sense of mirror symmetry.

Even in the case where the total space is not a Calabi--Yau manifold, 
the real Fourier--Mukai transform can also work. 
Actually, Lee and Leung \cite{LL} computed the real Fourier--Mukai transform of an associative 
and a coassociative cycle in a $G_{2}$-manifold 
and of a Cayley cycle in a ${\rm Spin}(7)$-manifold. 
In \cite{LL}, they picked some properties which 
the real Fourier--Mukai transform satisfies and called 
such a connection a \emph{deformed Donaldson--Thomas (dDT) connection}. 
In this paper, we suggest an alternative definition of a dDT connection 
for a manifold with a $\Sp$-structure which seems to be more appropriate 
by carefully computing the real Fourier--Mukai transform again.

As the real Fourier--Mukai transform of a submanifold written as a graph of a section of a trivial $T^{4}$-fibration over a flat 4-dimensional base $B$, we obtain the following.

\begin{theorem}[Theorem \ref{thm:FM Spin7}] \label{thm:FM Spin7-1}
Let $B\subset\mathbb{R}^4$ be an open set and 
$f:B\to T^{4}$ be a smooth function. 
Denote by $S=\{\,(x,f(x))\mid x\in B\,\}$ 
the graph of $f$, a 4-dimensional submanifold in $X=B\times T^{4}$. 
By the real Fourier--Mukai transform, 
$S$ corresponds to a Hermitian connection $\n^S$ 
of a trivial complex line bundle over $B\times (T^{4})^* \cong X$. 
Denote by $F_\n^S \in \i \Om^2 (X)$ the curvature 2-form of $\n^S$. 

Then, the graph $S$ is a Cayley submanifold with an appropriate orientation 
if and only if 
\[
\pi^2_7 \left( F_\nabla^S + \frac{1}{6} * (F_\nabla^S)^3 \right) =0
\quad\mbox{and}\quad  \pi^4_7  \left((F_\nabla^S)^{2} \right)=0.
\]
Here, $\pi^k_\l : \Om^k \rightarrow \Om^k_\l$ is the projection and 
$\Om^k_\l \subset \Om^k$ is the subspace of the space of $k$-forms 
corresponding to the $\l$-dimensional irreducible representation of $\Sp$ as in Subsection \ref{sec:Spin7 geometry}. 
\end{theorem}

We also compute the real Fourier--Mukai transform of 
a Cayley cycle, a Cayley submanifold with an ASD connection over it, 
and show the following.

\begin{theorem}[Theorem \ref{thm:FM Spin7 conn}] \label{thm:FM Spin7 conn-1}
Let $B\subset\mathbb{R}^4$ be an open set and 
$f:B\to T^{4}$ be a smooth function. 
Denote by $S=\{\,(x,f(x))\mid x\in B\,\}$ 
the graph of $f$, a 4-dimensional submanifold in $X=B\times T^{4}$. 
Let $\n^B$ be a Hermitian connection 
of a trivial complex line bundle $B \times \C \to B$. 
Denote by $F_\n^B \in \i \Om^2(B)$ the curvature of $\n^B$.

By the real Fourier--Mukai transform, 
the pair $(S, \n^B)$ corresponds to a Hermitian connection $\n$ 
of a trivial complex line bundle over $B\times (T^{4})^* \cong X$. 
Denote by $F_\n \in \i \Om^2 (X)$ the curvature 2-form of $\n$.  
Then, the following conditions are equivalent. 
\begin{enumerate}
\item
The graph $S$ is a Cayley submanifold with an appropriate orientation 
and if we identify $- \i F_\n^B \in \Om^2(B)$ with a 2-form on $S$, 
it is anti-self-dual with respect to the induced metric 
and the orientation which makes $S$ Cayley. 
\item
The Hermitian connection $\nabla$ satisfies 
\[
\pi^2_7 \left( F_\nabla + \frac{1}{6} * F_\nabla^3 \right) =0
\quad\mbox{and}\quad  \pi^4_7  \left(F_\nabla^{2} \right)=0.
\]
\end{enumerate}
\end{theorem}

Based on these theorems, we suggest the following definition. 

\begin{definition} \label{def:Spin7dDTMain}
Let $X^8$ be an 8-manifold with a ${\rm Spin}(7)$-structure 
$\Phi \in \Om^4$ 
and $L \to X$ be a smooth complex line bundle with a Hermitian metric $h$.
Denote by $\Om^k_\l \subset \Om^k$ the subspace of the space of $k$-forms 
corresponding to the $\l$-dimensional irreducible representation of $\Sp$ as in Subsection \ref{sec:Spin7 geometry}. 
Let $\pi^k_\l : \Om^k \rightarrow \Om^k_\l$ be the projection. 
A Hermitian connection $\nabla$ of $(L,h)$ satisfying 
\begin{align}\label{eq:Spin7dDT intro}
\pi^2_{7} \left( F_\nabla + \frac{1}{6} * F_\nabla^3 \right) = 0 
\quad \mbox{and} \quad 
\pi^{4}_{7}(F_{\nabla}^2) &= 0
\end{align}
is called a \emph{deformed Donaldson--Thomas connection 
for a manifold with a ${\rm Spin}(7)$-structure (a $\Sp$-dDT connection)}. 
Here, we regard the curvature 2-form $F_\n$ of $\n$ as a $\i \R$-valued closed 2-form on $X$. 
\end{definition}

In this paper, we post some evidences showing that Definition \ref{def:Spin7dDTMain} 
we suggest for a ${\rm Spin}(7)$-manifold 
is compatible with dDT connections for a $G_2$-manifold and 
dHYM connections for a Calabi-Yau 4-manifold 
in Lemmas \ref{lem:Spin7 G2} and \ref{lem:Spin7 CY4}.

We also compute the real Fourier--Mukai transform 
of (co)associative cycles in $G_2$-manifolds. 
This makes us confirm the definition of 
deformed Donaldson--Thomas connections 
for a manifold with a $G_2$-structure introduced by Lee and Leung \cite{LL}. 
This is also useful in the computation 
of the real Fourier--Mukai transform of Cayley cycles. 
It turns out that 
the real Fourier--Mukai transform of an associative cycle 
coincides with that of a coassociative cycle as stated in \cite{LL}. 
Moreover, the real Fourier--Mukai transform implies 
identities mirror to associator and Cayley equalities.  
In \cite{KYvol}, we show them 
and dDT connections for $G_2$- and $\Sp$-manifolds 
minimize a kind of the volume functional, which is called the Dirac-Born-Infeld (DBI) action in physics. 
%%%%%%%%%%%%%%%%%%%%%%%%%%%%%%%%%%%%%%%%%%%%%%%%%%%%%%%%%

This paper is organized as  follows. 
In Section \ref{sec:FM}, we explain the real Fourier--Mukai transform in detail. 
Section \ref{sec:basic} gives basic identities and some decompositions 
of the spaces of differential forms in $G_2$- and ${\rm Spin}(7)$-geometry 
that are used in this paper. 
In Section \ref{sec:dDT G2}--\ref{sec:dDT G2 2} 
we give computations of real Fourier--Mukai transforms 
and show 
Theorems \ref{thm:FM Spin7-1} and \ref{thm:FM Spin7 conn-1}. 
In Section \ref{sec:suggestSpin7dDT}, we 
show compatibilities of our ${\rm Spin}(7)$-dDT connections 
with dDT connections for a $G_2$-manifold and dHYM connections for a Calabi-Yau 4-manifold. 
In Appendix \ref{app:notation}, 
we summarize the notation used in this paper. 

\vspace{0.5cm}
\noindent{{\bf Acknowledgements}}: 
The authors would like to thank anonymous referees for the careful reading 
of an earlier version of this paper and many useful comments which helped to improve the quality of the paper.

%%%%%%%%%%%%%%%%%%%%%%%%%%%%%%%%%%%%%%%%%%%%%%%%%%%%%%%%%
%%%%%%%%%%%%%%%%%%%%%%%%%%%%%%%%%%%%%%%%%%%%%%%%%%%%%%%%%
%%%%%%%%%%%%%%%%%%%%%%%%%%%%%%%%%%%%%%%%%%%%%%%%%%%%%%%%%
%%%%%%%%%%%%%%%%%%%%%%%%%%%%%%%%%%%%%%%%%%%%%%%%%%%%%%%%%
%%%%%%%%%%%%%%%%%%%%%%%%%%%%%%%%%%%%%%%%%%%%%%%%%%%%%%%%%
%%%%%%%%%%%%%%%%%%%%%%%%%%%%%%%%%%%%%%%%%%%%%%%%%%%%%%%%%
\section{The real Fourier--Mukai transform} \label{sec:FM}
In this section, we explain the real Fourier--Mukai transform. 
We need the following two fundamental facts. 
The first one is that a representation $\rho:\pi_{1}(M)\to GL(k,\mathbb{R})$ naturally 
assigns a flat connection $\tilde{\nabla}$ of $\mathbb{R}^{k}$-bundle $E$ over a manifold $M$ by 
\[E:=\tilde{M}\times_{\rho}\mathbb{R}^{k}:=(\tilde{M}\times \mathbb{R}^{k})/{\sim}, \]
where $\tilde{M}$ is the universal cover of $M$ and $(x,v)\sim 
(x \cdot \gamma, \rho(\gamma)^{-1}v)$ 
for $\gamma\in \pi_{1}(M)$. 
The flat connection $\tilde{\nabla}$ of $E$ is induced from the exterior derivative $d$ on $\tilde{M}\times \mathbb{R}^{k}$. 
The second one is that an $n$-dimensional torus $T^{n}$ $(=\mathbb{R}^{n}/(2\pi\mathbb{Z})^{n})$ is canonically 
isomorphic to 
\[\mathop{\mathrm{Hom}}(\pi_{1}((T^{n})^{*}),U(1))=\mathop{\mathrm{Hom}}(((2 \pi \Z)^{n})^{*},U(1)), \]
the set of all homomorphisms from the first fundamental group of its dual torus 
$(T^{n})^{*}$ $(=(\mathbb{R}^{n})^{*}/((2\pi\mathbb{Z})^{n})^{*})$ 
to $U(1)$ by 
\[T^{n}\ni a=[\tilde{a}]\mapsto \rho_{a}:=e^{-\sqrt{-1}\langle\,\cdot\,,\tilde{a}\rangle}\in \mathop{\mathrm{Hom}}(\pi_{1}((T^{n})^{*}),U(1)), \]
where $\langle\,\cdot\,,\cdot\, \rangle : (\R^n)^* \times \R^n \to \R$ 
is a dual pairing. 
Then, combining these two facts with $M=(T^{n})^{*}$, we see that 
a point $a$ in $T^{n}$ assigns a flat Hermitian connection $\tilde{\nabla}^{a}$ 
of a complex line bundle $E_{a}:=(\R^n)^{*}\times_{\rho_{a}}\mathbb{C}$ with the standard Hermitian metric 
over the dual torus $(T^{n})^{*}$. 
Actually, $\pi_{a}:E_{a}\to (T^{n})^{*}$ is isomorphic to the trivial $\mathbb{C}$-bundle $\pi_{0}:\underline{\mathbb{C}}\to (T^{n})^{*}$ since 
we have a nonvanishing section 
$s(y):=[\tilde y,e^{\sqrt{-1}\langle\, \tilde{y},\tilde{a}\rangle} \cdot 1]$ of $E_{a}$,
where 
$\tilde y \in (\R^n)^*$ representing 
$y \in (T^n)^* =(\mathbb{R}^{n})^{*}/((2\pi\mathbb{Z})^{n})^{*}$ and 
$1$ is the trivial section of $(\mathbb{R}^{n})^{*}\times \mathbb{C}$. 
The bundle isomorphism $\xi: E_{a}\to \underline{\mathbb{C}}$ is given, on each fiber, by 
\[\pi_{a}^{-1}(y)\ni c\cdot s(y)\mapsto c \in \mathbb{C}\,\,(=\pi_{0}^{-1}(y)). \]
Then, a flat Hermitian connection of $\underline{\mathbb{C}}$ is induced from $\tilde{\nabla}^{a}$ of $E_{a}$ and denote it by $\nabla^{a}$. 
The connection 1-form of $\nabla^{a}$ with respect to the section 
$1\in \Gamma((T^{n})^{*}, \underline{\mathbb{C}})$ 
is represented as 
\[
\begin{aligned}
\nabla^{a}1=\xi^{-1}(\tilde{\nabla}^{a}(\xi(1)))
=&e^{-\sqrt{-1}\langle\,\cdot\,,\tilde{a}\rangle}d(e^{\sqrt{-1}\langle\,\cdot\,,\tilde{a}\rangle}\cdot 1)\\
=&\left(\sqrt{-1}d\langle\,\cdot\,,\tilde{a}\rangle\right)\otimes 1. 
\end{aligned}
\]
In summary, a point $a=[(a^{1},\cdots,a^{n})]\in T^{n}$ assigns an equivalence class of a Hermitian complex line bundle 
with a flat connection over $(T^{n})^{*}$ and 
one of its representatives is the trivial $\mathbb{C}$-bundle with the standard Hermitian metric and a flat Hermitian connection $\nabla^{a}$ defined  by
\[\nabla^{a}:=d+\sqrt{-1}\sum_{i=1}^{n}a^{i}dy^{i}, \]
where $y=(y^{1},\cdots,y^{n})$ are the standard coordinates on $(T^{n})^{*}$. 
This correspondence $a\mapsto \nabla^{a}$ 
is also explained in \cite[Section 3.2.1]{DK}.

When we consider the family of this correspondence, 
we get the real Fourier--Mukai transform. 
Precisely, let $B\subset \mathbb{R}^{k}$ be an open set with coordinates $x=(x^{1},\cdots,x^{k})$ and 
$f=(f^{1},\cdots,f^{n}):B\to T^{n}$ be a smooth map. 
Then, we get two objects: a submanifold and a connection. 
The $k$-dimensional submanifold in $X:=B\times T^{n}$, denoted by $S$, is defined as the graph of $f$, that is, 
\[S:=\{\,(x,f(x))\mid x\in B\,\}. \]
On the other hand, taking the family of $\nabla {}^{f(x)}$ for all $x\in B$, we get a Hermitian connection 
\[\nabla:=d+\sqrt{-1}\sum_{i=1}^{n}f^{i}dy^{i}\]
of the trivial $\mathbb{C}$-bundle over $X^{*}:=B\times (T^{n})^{\ast}$. 
We usually identify $B\times (T^{n})^{\ast}$ with $B\times T^{n}$. 
We call $\nabla$ \emph{the real Fourier--Mukai transform} of $S$. 
Basically, a property on $S$ is first interpreted as one of $f$ and second reinterpreted as one of $\nabla$. 
We remark that the real Fourier--Mukai transform of $(S,\nabla^{B})$, the pair of a graph of $f$ and a 
Hermitian connection $\nabla^{B} = d+\sqrt{-1}\sum_{i=1}^{k}A^{i}dx^{i}$ 
of the trivial $\mathbb{C}$-bundle over $B \cong S$, is also defined by 
\[\nabla:=d+\sqrt{-1}\sum_{i=1}^{k}A^{i}dx^{i}+\sqrt{-1}\sum_{i=1}^{n}f^{i}dy^{i}\]
as a Hermitian connection of the trivial $\mathbb{C}$-bundle over $X^{*}=B\times (T^{n})^{\ast}$. 

%%%%%%%%%%%%%%%%%%%%%%%%%%%%%%%%%%%%%%%%%%%%%%%%%%%%%%%%%
%%%%%%%%%%%%%%%%%%%%%%%%%%%%%%%%%%%%%%%%%%%%%%%%%%%%%%%%%
%%%%%%%%%%%%%%%%%%%%%%%%%%%%%%%%%%%%%%%%%%%%%%%%%%%%%%%%%
%%%%%%%%%%%%%%%%%%%%%%%%%%%%%%%%%%%%%%%%%%%%%%%%%%%%%%%%%
%%%%%%%%%%%%%%%%%%%%%%%%%%%%%%%%%%%%%%%%%%%%%%%%%%%%%%%%%
%%%%%%%%%%%%%%%%%%%%%%%%%%%%%%%%%%%%%%%%%%%%%%%%%%%%%%%%%
\section{Basics on $G_2$- and ${\rm Spin}(7)$-geometry} \label{sec:basic}
In this section, 
we collect some basic definitions and equations on 
$G_2$- and ${\rm Spin}(7)$-geometry 
which we need in the calculations
in this paper for the reader's convenience. 
See for example \cite{Bryant, HL, Kar} for references.

%%%%%%%%%%%%%%%%%%%%%%%%%%%%%%%%%%%%%%%%%%%%%%%%%%%%%%%%%
\subsection{The Hodge-$\ast$ operator}
Let $V$ be an $n$-dimensional oriented real vector space with an inner product $g$. 
Denote by $\la \cdot, \cdot \ra$ the 
induced inner product on $\Lambda^k V^*$ from $g$. 
Let $\ast$ be the Hodge-$\ast$ operator.
The following identities are frequently used throughout this paper. 

For $\alpha, \beta \in \Lambda^k V^*$ and $v \in V$, we have 
\[
\begin{aligned}
\ast^2|_{\Lambda^k V^*} &= (-1)^{k(n-k)} {\rm id}_{\Lambda^k V^*}, & 
\la \ast \alpha, \ast \beta \ra &= \la \alpha, \beta \ra, \\ 
i(v) \ast \alpha &= (-1)^k \ast (v^\flat \wedge \alpha), & 
\ast( i(v) \alpha) &= (-1)^{k+1} v^\flat \wedge \ast \alpha. 
\end{aligned}
\]

%%%%%%%%%%%%%%%%%%%%%%%%%%%%%%%%%%%%%%%%%%%%%%%%%%%%%%%%%

%%%%%%%%%%%%%%%%%%%%%%%%%%%%%%%%%%%%%%%%%%%%%%%%%%%%%%%%%
%%%%%%%%%%%%%%%%%%%%%%%%%%%%%%%%%%%%%%%%%%%%%%%%%%%%%%%%%
%%%%%%%%%%%%%%%%%%%%%%%%%%%%%%%%%%%%%%%%%%%%%%%%%%%%%%%%%
%%%%%%%%%%%%%%%%%%%%%%%%%%%%%%%%%%%%%%%%%%%%%%%%%%%%%%%%%
%%%%%%%%%%%%%%%%%%%%%%%%%%%%%%%%%%%%%%%%%%%%%%%%%%%%%%%%%
%%%%%%%%%%%%%%%%%%%%%%%%%%%%%%%%%%%%%%%%%%%%%%%%%%%%%%%%%

\subsection{Basics on $G_2$-geometry} \label{sec:G2 geometry}
Let $V$ be an oriented $7$-dimensional vector space. A \emph{$G_2$-structure} on $V$ is 
a 3-form $\varphi \in \Lambda^3 V^*$ such that there is a positively oriented basis 
$\{\, e_i \,\}_{i=1}^7$ of $V$ 
with the dual basis $\{\, e^i \,\}_{i=1}^7$ of $V^\ast$ satisfying 
\begin{equation} \label{varphi}
\varphi = e^{123} + e^{145} + e^{167} + e^{246} - e^{257} - e^{347} - e^{356},
\end{equation}
where $e^{i_1 \cdots i_k}$ is short for $e^{i_1} \wedge \cdots \wedge e^{i_k}$. Setting $\vol := e^{1 \cdots 7}$, 
the 3-form $\varphi$ uniquely determines an inner product $g_\varphi$ via 
\begin{align} \label{eq:form-1def}
g_\varphi(u,v)\; \vol = \dfrac16 i(u) \varphi \wedge i(v) \varphi \wedge \varphi 
\end{align}
for $u,v \in V$. 
It follows that any oriented basis $\{\, e_i \,\}_{i=1}^7$ for which \eqref{varphi} holds is orthonormal with respect to $g_\varphi$. Thus, the Hodge-dual of $\varphi$ with respect to $g_\varphi$ is given by 
\begin{equation} \label{varphi*}
\ast \varphi = e^{4567} + e^{2367} + e^{2345} + e^{1357} - e^{1346} - e^{1256} - e^{1247}.
\end{equation}
The stabilizer of $\varphi$ is known to be the exceptional $14$-dimensional simple Lie group 
$G_2 \subset {\rm GL}(V)$. The elements of $G_2$ preserve both $g_\varphi$ and $\vol$, that is, 
$G_2 \subset {\rm SO}(V, g_\varphi)$.

We summarize important well-known facts about the decomposition of tensor products of $G_2$-modules into irreducible summands. 
Denote by $V_k$ the $k$-dimensional irreducible $G_2$-module if there is a unique such module. For instance, $V_7$ is the irreducible $7$-dimensional $G_2$-module $V$ from above, 
and $V_7^* \cong V_7$. For its exterior powers, we obtain the decompositions
\begin{equation} \label{eq:DiffForm-V7}
\begin{array}{rlrl}
\Lambda^0 V^* \cong \Lambda^7 V^* \cong V_1, \quad
& \Lambda^2 V^*  \cong \Lambda^5 V^* \cong V_7 \oplus  V_{14},\\[2mm]
\Lambda^1 V^* \cong \Lambda^6 V^* \cong V_7, \quad
& \Lambda^3 V^* \cong \Lambda^4 V^* \cong V_1 \oplus V_7 \oplus V_{27},
\end{array}
\end{equation}
where $\Lambda^k V^* \cong \Lambda^{7-k} V^*$ due to the $G_2$-invariance of the Hodge isomorphism $\ast: \Lambda^k V^* \to \Lambda^{7-k} V^*$. We denote by $\Lambda^k_\l V^* \subset \Lambda^k V^*$ the subspace isomorphic to $V_\l$. 
Let 
\[
\pi^k_\l: \Lambda^k V^* \rightarrow \Lambda^k_\l V^*
\]
be the canonical projection. 
Identities for these spaces we need in this paper are as follows. 
\begin{equation*}
\begin{aligned}
\Lambda^2_7 V^* =& \{\, i(u) \varphi \mid u \in V \,\} 
= \{\, \alpha \in \Lambda^2 V^* \mid * (\varphi \wedge \alpha) = 2 \alpha \,\},\\
\Lambda^2_{14} V^* =& \{\, \alpha \in \Lambda^2 V^* \mid \ast \varphi \wedge \alpha = 0\,\} 
= \{\, \alpha \in \Lambda^2 V^* \mid * (\varphi \wedge \alpha) = - \alpha \,\},\\
\Lambda^3_1 V^* =& \R \varphi, \\
\Lambda^3_7 V^* =& \{\, i(u) * \varphi \in \Lambda^3 V^* \mid u \in V \,\}. 
\end{aligned}
\end{equation*}
The following equations are well-known and useful in this paper. 

\begin{lemma} \label{lem:G2 identities}
For any $u \in V$, we have the following identities. 
\[
\begin{aligned}
\varphi \wedge i(u) * \varphi &= -4 * u^{\flat}, 
\\
* \varphi \wedge i(u) \varphi &= 3 * u^{\flat}, \\
\varphi \wedge i(u) \varphi &= 2 * (i(u) \varphi) = 2 u^{\flat} \wedge * \varphi. 
\end{aligned}
\]
\end{lemma}

\begin{definition} \label{def:G2mfd}
Let $X$ be an oriented 7-manifold. A \emph{$G_2$-structure} on $X$ is a $3$-form $\varphi \in \Om^3$ 
such that at each $p \in X$ there is a positively oriented basis 
$\{\, e_i \,\}_{i=1}^7$ of $T_p X$ such that $\varphi_p \in \Lambda^3 T^*_p X$ is of the form \eqref{varphi}. As noted above, $\varphi$ determines a unique Riemannian metric $g = g_\varphi$ on $X$ by \eqref{eq:form-1def}, 
and the basis $\{\, e_i \,\}_{i=1}^7$ is orthonormal with respect to $g$.
A $G_2$-structure $\varphi$ is called \emph{torsion-free} if 
it is parallel with respect to the Levi-Civita connection of $g=g_\varphi$. 
A manifold with a torsion-free $G_2$-structure is called a \emph{$G_2$-manifold}. 
\end{definition}

A manifold $X$ admits a $G_2$-structure if and only if 
its frame bundle is reduced to a $G_2$-subbundle. 
Hence, considering its associated subbundles, 
we see that 
$\Lambda^* T^* X$ has the same decomposition as in \eqref{eq:DiffForm-V7}. 
The algebraic identities above also hold.

\subsection{Associative and coassociative submanifolds}
On a $G_2$-manifold $(X, \varphi)$, the $G_2$-structure $\varphi$ and its Hodge dual $* \varphi$ 
are known to be calibrations. 
The corresponding calibrated submanifolds are called \emph{associative submanifolds} 
and \emph{coassociative submanifolds}, respectively. 
By \cite{HL, McLean}, we can characterize these submanifolds as follows. 

\begin{lemma} \label{lem:asso coasso char}
An oriented 3-dimensional submanifold $A \subset X$ is associative with an appropriate orientation 
if and only if $* \varphi (v_1, v_2, v_3, \cdot) =0$ for any $p \in A$ and $v_1, v_2, v_3 \in T_p S$. 
An oriented 4-dimensional submanifold $C \subset X$ is coassociative with an appropriate orientation 
if and only if the restriction of $\varphi$ to $C$ vanishes. 
\end{lemma}

%%%%%%%%%%%%%%%%%%%%%%%%%%%%%%%%%%%%%%%%%%%%%%%%%%%%%%%%%
\subsection{${\rm Spin}(7)$-geometry} \label{sec:Spin7 geometry}
Let $W$ be an $8$-dimensional oriented real vector space. 
A \emph{$\Sp$-structure on $W$} is 
a 4-form $\Phi \in \Lambda^4 W^*$ such that there is a positively oriented basis 
$\{\, e_i \,\}_{i=0}^7$ of $W$ 
with dual basis $\{\, e^i \,\}_{i=0}^7$ of $W^\ast$ satisfying 
\begin{equation}\label{Phi4}
\begin{aligned}
\Phi := & e^{0123} + e^{0145} + e^{0167} + e^{0246} - e^{0257} - e^{0347} - e^{0356}\\
& + e^{4567} + e^{2367} + e^{2345} + e^{1357} - e^{1346} - e^{1256} - e^{1247}, 
\end{aligned}
\end{equation}
where $e^{i_1 \cdots i_k}$ is short for $e^{i_1} \wedge \cdots \wedge e^{i_k}$.
Defining forms 
$\varphi$ and $\ast_7 \varphi$ on $V := \rmspan \{\, e_i \,\}_{i=1}^7 \subset W$ 
as in \eqref{varphi} and \eqref{varphi*}, where $*_7$ stands for the Hodge star operator on $V$, 
we have 
\[
\Phi = e^0 \wedge \varphi + \ast_7 \varphi. 
\]
Note that $\Phi$ is self-dual, that is, $*_8 \Phi = \Phi$, where 
$*_8$ is the Hodge star operator on $W$. 
It is known that $\Phi$ uniquely determines an inner product $g_\Phi$ and a volume form 
and the subgroup of ${\rm GL}(W)$ preserving $\Phi$ is isomorphic to $\Sp$. 
As in Definition \ref{def:G2mfd}, we can define 
an 8-manifold with a $\Sp$-structure and a $\Sp$-manifold.

Denote by $W_k$ the $k$-dimensional irreducible $\Sp$-module if there is a unique such module. 
For example, $W_8$ is the irreducible $8$-dimensional $\Sp$-module from above, and $W_8^* \cong W_8$. 
The group $\Sp$ acts irreducibly 
on $W_7 \cong \R^7$ as the double cover of ${\rm SO}(7)$. 
For its exterior powers, we obtain the decompositions 
\begin{equation*}
\begin{aligned}
\Lambda^0 W^* &\cong \Lambda^8 W^* \cong W_1, \quad
& \Lambda^2 W^*  \cong \Lambda^6 W^* &\cong W_7 \oplus  W_{21},\\
\Lambda^1 W^* &\cong \Lambda^7 W^* \cong W_8, \quad
& \Lambda^3 W^* \cong \Lambda^5 W^* &\cong W_8 \oplus W_{48},\\
\Lambda^4 W^* &\cong W_1 \oplus W_7 \oplus W_{27} \oplus W_{35}
\end{aligned}
\end{equation*}
where $\Lambda^k W^* \cong \Lambda^{8-k} W^*$ due to the $\Sp$-invariance of the Hodge isomorphism $\ast_8: \Lambda^k W^* \to \Lambda^{8-k} W^*$. Again, we denote by $\Lambda^k_\l W^* \subset \Lambda^k W^*$ the subspace isomorphic to $W_\l$ in the above notation.

The space $\Lambda^k_7 W^*$ for $k = 2,4,6$ 
is explicitly given as follows. 
For the explicit descriptions of the other irreducible summands, 
see for example \cite[(4.7)]{KLS}. 

\begin{lemma} \label{lem:lambdas}
Let $e^0 \in W^*$ be a unit vector. 
Set $V^* = (\R e^0)^\perp$, the orthogonal complement of $\R e^0$. 
The group $\Sp$ acts irreducibly on $V^*$ as the double cover of ${\rm SO}(7)$, 
and hence, we have the identification $V^* \cong W_7$. 
Then, the following maps are $\Sp$-equivariant isometries. 
\begin{equation} \label{def:lambda}
\lambda^k: V^* \longrightarrow \Lambda^k_7 W^*, \quad 
\begin{array}{ll} 
\lambda^2(\alpha) := \frac{1}{2} \left( e^0 \wedge \alpha + i (\alpha^\sharp) \varphi \right), \\[2mm]
\lambda^4(\alpha) := 
\frac{1}{\sqrt{8}} \left( e^0 \wedge i (\alpha^\sharp) \ast_7 \varphi - \alpha \wedge \varphi \right), \\[2mm]
\lambda^6(\alpha) := \frac{1}{3} \Phi \wedge \lambda^2(\alpha) = \ast_8 \lambda^2(\alpha). 
\end{array}
\end{equation}
Here, $\ast_8$ and $\ast_7$ are the Hodge star operators on $W^*$ and $V^*$, respectively.
\end{lemma}

\begin{proof}
The maps above are $\Sp$-equivariant isomorphism by \cite[Lemma 4.2]{KLS}. 
We show that these are isometries. 
For $\alpha \in V^*$, we compute 
\[
4 | \lambda^2 (\alpha) |^2
= 
\la e^0 \wedge \alpha + i (\alpha^\sharp) \varphi, 
e^0 \wedge \alpha + i (\alpha^\sharp) \varphi \ra
= 
|\alpha|^2 + | i (\alpha^\sharp) \varphi |^2. 
\]
By Lemma \ref{lem:G2 identities}, we have 
\begin{equation} \label{eq:G2 norm}
| i (\alpha^\sharp) \varphi |^2 
=
*_7 \left(
i (\alpha^\sharp) \varphi \wedge *_7 (i (\alpha^\sharp) \varphi) \right)
=
*_7 \left(
i (\alpha^\sharp) \varphi \wedge \alpha \wedge *_7 \varphi \right)
=
3 |\alpha|^2. 
\end{equation}
Thus, we see that $\lambda^2$ is an isometry. 
By the definition of $\lambda^6$, this is also an isometry. 
We also compute 
\[
\begin{aligned}
8 | \lambda^4 (\alpha) |^2
&= 
\left \la e^0 \wedge i (\alpha^\sharp) \ast_7 \varphi - \alpha \wedge \varphi, 
e^0 \wedge i (\alpha^\sharp) \ast_7 \varphi - \alpha \wedge \varphi  \right \ra \\
&= 
| i (\alpha^\sharp) \ast_7 \varphi |^2 +  |\alpha \wedge \varphi|^2 \\
&= 2 |\alpha \wedge \varphi|^2. 
\end{aligned}
\]
The last term is computed as 
\[
\begin{aligned}
|\alpha \wedge \varphi|^2 
=&
\la \varphi, i(\alpha^\sharp) (\alpha \wedge \varphi) \ra\\
=&
\la \varphi, |\alpha|^2 \varphi - \alpha \wedge i(\alpha^\sharp) \varphi \ra
=
|\alpha|^2 |\varphi|^2 - |i(\alpha^\sharp) \varphi|^2
=
4 |\alpha|^2, 
\end{aligned}
\]
where we use $|\varphi|^2 = 7$ and \eqref{eq:G2 norm}. 
Hence, we see that $\lambda^4$ is an isometry. 
\end{proof}

We give a relation between $*_8$ and $*_7$, which is useful in Section \ref{sec:dDT Spin7}. 

\begin{lemma}
For $\alpha \in \Lambda^k V^*$, we have 
\begin{equation}\label{eq:Hodge 78}
\begin{aligned} 
*_8 \alpha = (-1)^k e^0 \wedge *_7 \alpha, \quad 
*_7 \alpha = *_8 (e^0 \wedge \alpha). 
\end{aligned}
\end{equation}
\end{lemma}

\begin{proof}
Denote by ${\rm vol}_7$ the volume form on $V^*$. 
The volume form on $W^*$ is given by $e^0 \wedge {\rm vol}_7$. 
Then, for any $\beta \in \Lambda^k V^*$, we have 
\[
\beta \wedge *_8 \alpha 
= \la \beta, \alpha \ra e^0 \wedge {\rm vol}_7
=e^0 \wedge \beta \wedge *_7 \alpha
= (-1)^k \beta \wedge e^0 \wedge *_7 \alpha, 
\]
which implies the first equation. 
The second equation follows from the first. 
\end{proof}

We give some formulas about projections onto  some irreducible summands. 
Denote by 
\begin{align} \label{eq:proj}
\pi^k_\l : \Lambda^k W^* \rightarrow \Lambda^k_\l W^*
\end{align}
the canonical projection. 
When $k=2, 4, 6$ and $\l=7$, 
Lemma \ref{lem:lambdas} implies that 
\begin{equation} \label{eq:Spin7 proj 1}
\pi^k_\l (\alpha^k) = \sum_{\mu=1}^7 \la \alpha^k, \lambda^k (e^\mu) \ra \cdot \lambda^k (e^\mu)
\end{equation}
for $\alpha^k \in \Lambda^k W^*$, 
where $\{\, e^\mu\,\}_{\mu=1}^7$ is an orthonormal basis of $V^*$. 

We give other descriptions of $\pi^k_\l$ for $k=2,6$. 
Since the map 
$\Lambda^2 W^* \ni \alpha^2 \mapsto *_8 (\Phi \wedge \alpha^2) \in \Lambda^2 W^*$ 
is $\Sp$-equivariant, the simple computation and Schur's lemma give the following: 
\[
\begin{aligned}
\Lambda^2_7 W^* 
=& \{\, \alpha^2 \in \Lambda^2 W^* \mid \Phi \wedge \alpha^2 = 3 *_8 \alpha^2 \,\}, \\
\Lambda^2_{21} W^*
=& 
\{\, \alpha^2 \in \Lambda^2 W^* \mid \Phi \wedge \alpha^2 = - *_8 \alpha^2 \,\}.  
\end{aligned}
\]
Since 
$\alpha^2 = \pi^2_7 (\alpha^2) + \pi^2_{21} (\alpha^2)$ and 
$*_8 (\Phi \wedge \alpha^2) = 3 \pi^2_7 (\alpha^2) - \pi^2_{21} (\alpha^2)$
for a 2-form $\alpha^2 \in \Lambda^2 W^*$, it follows that 
\begin{equation} \label{eq:Spin7 proj 2}
\pi^2_7 (\alpha^2) = \frac{\alpha^2 + *_8 (\Phi \wedge \alpha^2)}{4}, \quad
\pi^2_{21} (\alpha^2) = \frac{3 \alpha^2 - *_8 (\Phi \wedge \alpha^2)}{4}. 
\end{equation}
Since 
$\ast_8: \Lambda^6_\l W^* \to \Lambda^2_\l W^*$ is an isomorphism, 
we also obtain for a 6-form $\alpha^6 \in \Lambda^6 W^*$
\begin{equation} \label{eq:Spin7 proj 3}
\pi^6_7 (\alpha^6) = \frac{\alpha^6 + \Phi \wedge *_8 \alpha^6}{4}, \quad
\pi^6_{21} (\alpha^6) = \frac{3 \alpha^6 - \Phi \wedge *_8 \alpha^6}{4}. 
\end{equation}

%%%%%%%%%%%%%%%%%%%%%%%%%%%%%%%%%%%%%%%%%%%%%%%%%%%%%%%%%
\subsection{Cayley submanifolds}
The 4-form $\Phi$ given by \eqref{Phi4} is known to be a calibration. 
The corresponding calibrated submanifold is called a \emph{Cayley submanifold}. 
We give a characterization of Cayley submanifolds, 
which is equivalent to that of \cite{HL, McLean} by Lemma \ref{lem:lambdas}. 

Define a $\Sp$-equivariant map 
$\tau: \Lambda^4 W \rightarrow \Lambda^4_7 W^*$ by 
\begin{equation} \label{eq:def tau}
\tau (u_0, u_1, u_2, u_3) = \pi^4_7 (u_0^\flat \wedge u_1^\flat \wedge u_2^\flat \wedge u_3^\flat). 
\end{equation}
If $\{\, e^\mu \,\}_{\mu=1}^7$ is an oriented orthonormal basis of $V^*$, 
\eqref{eq:Spin7 proj 1} implies that 
\[
\tau = \sum_{\mu=1}^7 \lambda^4(e^\mu) \otimes \lambda^4(e^\mu). 
\]

\begin{lemma} \label{lem:Cayley id}
For any $u_0, u_1, u_2, u_3 \in W$, we have
\[
|\Phi (u_0, u_1, u_2, u_3)|^2 + 8 |\tau (u_0, u_1, u_2, u_3)|^2 = |u_0 \wedge u_1 \wedge u_2 \wedge u_3 |^2. 
\]
\end{lemma}

\begin{proof}
We only have to show the equation when $\{\, u_0, u_1, u_2, u_3 \,\}$ is orthonormal. 
Since the both sides are $\Sp$-invariant and 
$\Sp$ acts transitively on ${\rm Gr}_3 (W)$, the Grassmannian of 3-planes in $W$, 
we may assume that $u_0 =e_0, u_1 = e_1$ and $u_2 = e_2$. 
Since the stabilizer at $\rmspan \{\, e_0, e_1, e_2 \,\}$ in $\Sp$ 
is the group ${\rm SU}(2)$ acting on the plane 
$\rmspan \{\, e_3, e_4, e_5, e_6, e_7 \,\} \cong \R \oplus \C^2$, 
we may assume that $u_3 = k e_3 + \l e_4$, where $k^2+\l^2 =1$. 
Then, \eqref{Phi4} implies that 
\[
|\Phi (u_0, u_1, u_2, u_3)|^2 = k^2. 
\]
By \eqref{def:lambda} and \eqref{varphi*}, we have 
\[
\begin{aligned}
\sqrt {8} \lambda^4 (e^\mu) (u_0, u_1, u_2, u_3) 
=&
*_7 \varphi (e_\mu, e_1, e_2, k e_3 + \l e_4)\\
=&
- *_7 \varphi (e_1, e_2, k e_3 + \l e_4, e_\mu)\\
=&
(e^{56} + e^{47}) (k e_3 + \l e_4, e_\mu)
=
\l \delta_{\mu 7}. 
\end{aligned}
\]
Then, we have 
\[
8 |\tau (u_0, u_1, u_2, u_3)|^2 
= 
8 \sum_{\mu=1}^7 |\lambda^4 (e^\mu) (u_0, u_1, u_2, u_3)|^2
= \l^2. 
\]
Since $|u_0 \wedge u_1 \wedge u_2 \wedge u_3 |^2 = k^2 + \l^2$, 
the proof is completed. 
\end{proof}

Lemma \ref{lem:Cayley id} immediately implies the following. 

\begin{lemma} \label{lem:Cayley char}
An oriented 4-dimensional submanifold $C \subset W$ is Cayley 
with an appropriate orientation if and only if 
the restriction of $\tau$ to $C$ vanishes. 
\end{lemma}

%%%%%%%%%%%%%%%%%%%%%%%%%%%%%%%%%%%%%%%%%%%%%%%%%%%%%%%%%
%%%%%%%%%%%%%%%%%%%%%%%%%%%%%%%%%%%%%%%%%%%%%%%%%%%%%%%%%
%%%%%%%%%%%%%%%%%%%%%%%%%%%%%%%%%%%%%%%%%%%%%%%%%%%%%%%%%
%%%%%%%%%%%%%%%%%%%%%%%%%%%%%%%%%%%%%%%%%%%%%%%%%%%%%%%%%
%%%%%%%%%%%%%%%%%%%%%%%%%%%%%%%%%%%%%%%%%%%%%%%%%%%%%%%%%
%%%%%%%%%%%%%%%%%%%%%%%%%%%%%%%%%%%%%%%%%%%%%%%%%%%%%%%%%
%%%%%%%%%%%%%%%%%%%%%%%%%%%%%%%%%%%%%%%%%%%%%%%%%%%%%%%%%
%%%%%%%%%%%%%%%%%%%%%%%%%%%%%%%%%%%%%%%%%%%%%%%%%%%%%%%%%
%%%%%%%%%%%%%%%%%%%%%%%%%%%%%%%%%%%%%%%%%%%%%%%%%%%%%%%%%
%%%%%%%%%%%%%%%%%%%%%%%%%%%%%%%%%%%%%%%%%%%%%%%%%%%%%%%%%
%%%%%%%%%%%%%%%%%%%%%%%%%%%%%%%%%%%%%%%%%%%%%%%%%%%%%%%%%
%%%%%%%%%%%%%%%%%%%%%%%%%%%%%%%%%%%%%%%%%%%%%%%%%%%%%%%%%
\section{The real Fourier--Mukai transform for coassociative $T^4$-fibrations} \label{sec:dDT G2}
In this section, we compute the real Fourier--Mukai transform of associative cycles. 
This makes us confirm the definition of 
deformed Donaldson--Thomas connections 
for a manifold with a $G_2$-structure 
introduced by Lee and Leung \cite{LL}. 
This is also useful in the computation of Section \ref{sec:dDT Spin7}.

Let $B\subset\mathbb{R}^3$ be an open set with coordinates $(x^{1},x^{2},x^{3})$ and 
$f=(f^{4},f^{5},f^{6},f^{7}):B\to T^{4}$ be a smooth function with values in $T^{4}$. 
We use coordinates $(y^{4},y^{5},y^{6},y^{7})$ for $T^{4}$. 
Put
\[S:=\{\,(x,f(x))\mid x\in B\,\}\]
the graph of $f$, a 3-dimensional submanifold in $X:=B\times T^{4}$. 
Set 
\[
\om_1=dy^{45} + dy^{67}, \quad \om_2=dy^{46} + dy^{75}, \quad \om_3=-(dy^{47} + dy^{56}). 
\] 
By \eqref{varphi} and \eqref{varphi*}, 
the standard $G_2$-structure $\varphi$ on $X$ and its Hodge dual $\ast \varphi$ are described as 
\begin{align}
\varphi &= dx^{123} + \sum_{i=1}^3 dx^i \wedge \om_i, \label{eq:FM asso} \\
\ast \varphi &= dy^{4567} + \sum_{k \in \Z/3} dx^{k,k+1} \wedge \om_{k+2}.  \label{eq:FM coasso}
\end{align}
Let 
\[
\n^B = d + \i \sum_{j=1}^3 A^j dx^j
\]
be a Hermitian connection of a trivial complex line bundle $B \times \C \to B$, 
where 
$A^j : B \to \R$ is a smooth function.

Next, we consider the mirror side. 
The real Fourier--Mukai transform of $(S, \n^B)$ is the connection on $X^{*}(\cong X)$ defined by 
\[
\nabla:=d+ \i \sum_{j=1}^3 A^{j} dx^j + \i \sum_{a=4}^7 f^{a} dy^{a}.
\]
Then, its curvature 2-form $F_\nabla$ is given by $F_\nabla = F_\n^B + F_\n^S$, where 
\begin{equation} \label{eq:FM curv}
F_\n^B = \i \sum_{i, j=1}^3 \frac{\partial A^{j}}{\partial x^{i}} dx^i \wedge dx^j, \quad 
F_\n^S = \i \sum_{i=1}^3 \sum_{a=4}^7 \frac{\partial f^{a}}{\partial x^{i}} dx^i \wedge dy^a. 
\end{equation}

We first describe the condition for $S$ to be associative in terms of $F_\n^S$ in Proposition \ref{prop:FM G2}. 
Using this, we show that the similar statement also holds for $F_\n$ in Proposition \ref{prop:FM G2 conn}. 

\begin{proposition} \label{prop:FM G2}
The following conditions are equivalent. 
\begin{enumerate}
\item
The graph $S$ is an associative submanifold with an appropriate orientation. 
\item
$(F_\n^S)^3/6 + F_\n^S \wedge \ast\varphi=0$. 
\item
$(F_\n^S)^3/6 + F_\n^S \wedge \ast\varphi=0
\quad\mbox{and}\quad \varphi\wedge \ast (F_\n^S)^{2}=0.$
\end{enumerate}
\end{proposition}

\begin{remark}
A similar statement for graphical submanifolds 
is given by Harvey and Lawson in \cite[Chapter I\hspace{-.1em}V,Theorem 2.4]{HL}. 
In terms of differential equations for $f$, 
they obtained two equations (2)' and (3)', which correspond to (2) and (3), respectively. 
Then, they stated that (1) and (2)' are equivalent in the theorem, 
and (3)' appeared only in the proof. 
They first showed that (1) and (3)' are equivalent. 
Using the assumption that $S$ is a graph, 
they showed that (2)' implies (1). 
Since (3)' obviously implies (2)', 
they obtained the equivalence. 
Actually, (2) and (3) are equivalent in general. 
See \cite[Remark 3.3]{KY}.

We can also consider the real Fourier--Mukai transform 
of a coassociative graph in associative $T^3$-fibrations. 
In Proposition \ref{prop:FM G2 2}, we show that we obtain the same equations as stated in \cite{LL}. 
\end{remark}

\begin{proof}
Since $(3) \Rightarrow (2)$ is obvious and the converse holds by \cite[Remark 3.3]{KY}, 
(2) and (3) are equivalent. 
We show the equivalence of (1) and (3). 
By Lemma \ref{lem:asso coasso char},  
$S$ is associative with an appropriate orientation if and only if  
$\ast \varphi (v_{1},v_{2},v_{3}, \,\cdot\,)=0$ for any $p\in S$ and $v_{1},v_{2},v_{3}\in T_{p}S$.
Set $\partial_{i}:=\partial/\partial x^i$ and $\partial_a:=\partial/\partial y^a$ 
for $1 \leq i \leq 3$ and $4 \leq a \leq 7$. 
Then, the tangent space of $S$ is spanned by $v_1,v_2,v_3$, where 
\[
v_{j}:=\p_j + \sum_{a=4}^7 \frac{\partial f^{a}}{\partial x^{j}} \p_a. 
\]
By \eqref{eq:FM curv}, note that 
\[
v_j^\flat = dx^j + i(\p_j) F
\]
where we set $F=- \i F^S_\nabla$. 
Since $ \ast \varphi (v_{1},v_{2},v_{3}, \,\cdot\,)=0$ is equivalent to 
$v_1^\flat \wedge v_2^\flat \wedge v_3^\flat \wedge \varphi =0$, we have 
\begin{equation}\label{123phi=0}
0 = \left( dx^1 + i (\p_1) F \right) \wedge \left( dx^2 + i (\p_2) F \right) 
\wedge \left( dx^3 + i (\p_3) F \right) \wedge \varphi. 
\end{equation}
Since $dx^{123} \wedge \varphi =0$, this is equivalent to 
\[
0 = I_1 + I_2 + I_3, 
\]
where 
\[
\begin{aligned}
I_1 &= \sum_{k \in \Z/3} dx^{k, k+1} \wedge i (\p_{k+2}) F \wedge \varphi, \\
I_2 &= \sum_{k \in \Z/3} dx^k \wedge i (\p_{k+1}) F \wedge i (\p_{k+2}) F \wedge \varphi, \\
I_3 &= i (\p_{1}) F \wedge i (\p_{2}) F \wedge i (\p_{3}) F \wedge \varphi. 
\end{aligned}
\]
Since $I_1$ and $I_3$ are linear combinations of $dx^{123} \wedge dy^{abc}$'s  
and $I_2$ is a linear combination of $dx^{i j} \wedge dy^{4567}$'s, 
$S$ is associative with an appropriate orientation if and only if 
\begin{equation} \label{eq:FM asso cond}
I_1 + I_3 =0, \quad I_2 =0. 
\end{equation}
Now, we compute $I_1, I_2$ and $I_3$. 
By \eqref{eq:FM asso}, we have 
\[
I_1 
= \sum_{k \in \Z/3} dx^{k, k+1} \wedge i (\p_{k+2}) F \wedge (dx^{k+2} \wedge \om_{k+2})
= -dx^{123} \wedge \sum_{k=1}^3 \om_k \wedge i (\p_k) F. 
\]
Since $i (\p_k) F$ is the linear combination of $dy^a$'s and 
$dx^{123} \wedge  \om^k \wedge F =0$ by \eqref{eq:FM curv}, we see that 
\[
I_1 = -dx^{123} \wedge \sum_{k=1}^3 i (\p_k) \left( \om_k \wedge F \right)
=
-\sum_{k=1}^3 \left( i (\p_k) dx^{123} \right) \wedge \om_k \wedge F. 
\]
Then, by \eqref{eq:FM coasso}, we obtain 
\begin{equation} \label{eq:FM I1}
I_1 = - * \varphi \wedge F. 
\end{equation}

Next, we compute $I_3$. 
Since $i (\p_k) F$ is the linear combination of $dy^a$'s, we see that 
\[
I_3 
= - dx^{123} \wedge i (\p_{1}) F \wedge i (\p_{2}) F \wedge i (\p_{3}) F 
\]
and 
\[
\begin{aligned}
i (\p_{3}) i (\p_{2}) i (\p_{1}) \left( \frac{1}{6} F^3 \right)
=& 
i (\p_{3}) i (\p_{2}) \left( \frac{1}{2} i (\p_{1})F \wedge F^2 \right)\\
=&
i (\p_{3}) \left( - i (\p_{1})F \wedge  i (\p_{2}) F \wedge F \right)\\
=&
- i (\p_{1}) F \wedge i (\p_{2}) F \wedge i (\p_{3}) F. 
\end{aligned}
\]
By \eqref{eq:FM curv}, $F^3$ is the linear combination of $dx^{123} \wedge dy^{abc}$'s, 
and hence, we obtain 
\begin{equation} \label{eq:FM I3}
I_3 = \frac{1}{6} F^3. 
\end{equation}

Finally, we compute $I_2$. 
By \eqref{eq:FM asso}, we have 
\[
dx^k \wedge \varphi = dx^k \wedge (dx^{k+1} \wedge \om_{k+1} + dx^{k+2} \wedge \om_{k+2}). 
\]
Since $i (\p_k) F$ is the linear combination of $dy^a$'s, we see that 
\[
i (\p_j) i (\p_i) \left( \frac{1}{2} F^2 \right) = - i(\p_i) F \wedge i(\p_j) F. 
\]
Then, it follows that 
\[
\begin{aligned}
I_2 =& 
- \sum_{k \in \Z/3} \left( dx^{k, k+1} \wedge \om_{k+1} 
+ dx^{k, k+2} \wedge \om_{k+2} \right)
\wedge 
i (\p_{k+2}) i (\p_{k+1}) \left( \frac{1}{2} F^2 \right) \\
=&
- \frac{1}{2} (I_{2,1} + I_{2,2}), 
\end{aligned}
\]
where 
\[
\begin{aligned}
I_{2,1}
=&
\sum_{k \in \Z/3} dx^{k, k+1} \wedge 
\om_{k+1} \wedge 
i (\p_{k+2}) i (\p_{k+1}) F^2 \\
=& 
\sum_{k \in \Z/3} 
i (\p_{k+2}) 
\left(
dx^{k, k+1} \wedge 
\om_{k+1} \wedge 
i (\p_{k+1}) F^2 
\right)
\end{aligned}
\]
and 
\[
\begin{aligned}
I_{2,2}
=&
\sum_{k \in \Z/3} dx^{k+1, k} \wedge 
\om_{k} \wedge 
i (\p_{k}) i (\p_{k+2}) F^2 \\
=& 
\sum_{k \in \Z/3} i (\p_{k+2}) 
\left(
dx^{k, k+1} \wedge 
\om_{k} \wedge 
i (\p_{k}) F^2 
\right). 
\end{aligned}
\]
Since $dx^{k, k+1} \wedge \om_{k+1} \wedge F^2 = 0$, which is an 8-form,  
it follows that 
\[
\begin{aligned}
I_{2,1}
=&
\sum_{k \in \Z/3} 
i (\p_{k+2}) \left(
- i (\p_{k+1}) (dx^{k, k+1}) \wedge 
\om_{k+1} \wedge F^2
\right)\\
=&
\sum_{k \in \Z/3} 
i (\p_{k+2}) \left(
dx^k \wedge \om_{k+1} \wedge F^2 
\right). 
\end{aligned}
\]
Similarly, we compute 
\[
\begin{aligned}
I_{2,2}
=& 
\sum_{k \in \Z/3} i (\p_{k+2}) 
\left(
- dx^{k+1} \wedge \om_{k} \wedge F^2 
\right). 
\end{aligned}
\]
Then, we obtain 
\[
\begin{aligned}
2 I_2 =& 
\sum_{k \in \Z/3} 
i (\p_{k+2}) 
\left( F^2  \wedge 
(- dx^{k} \wedge \om_{k+1} + dx^{k+1} \wedge \om_{k}) \right) \\
=& 
\sum_{k \in \Z/3} 
i (\p_{k}) 
\left( F^2 \wedge 
(- dx^{k+1} \wedge \om_{k+2} + dx^{k+2} \wedge \om_{k+1}) \right). 
\end{aligned}
\]
By \eqref{eq:FM coasso}, we have
$i(\p_k) * \varphi = dx^{k+1} \wedge \om_{k+2} - dx^{k+2} \wedge \om_{k+1}$, 
and hence, 
\[
2 I_2 = 
- \sum_{k=1}^3 
i (\p_{k}) 
\left( F^2 \wedge i(\p_k) * \varphi \right). 
\]
Since 
\[
\begin{aligned}
- F^2 \wedge i(\p_k) * \varphi 
&= - \la F^2, * (i(\p_k) * \varphi) \ra {\rm vol}\\
&= \la F^2, d x^k \wedge \varphi \ra {\rm vol}\\
&= dx^k \wedge \varphi \wedge * (F^2) 
=
\la \varphi \wedge * (F^2), * dx^k \ra {\rm vol}, 
\end{aligned}
\]
we see that 
$2 I_2 = \sum_{k=1}^3 \la \varphi \wedge * (F^2), * dx^k \ra * dx^k$. 
The equation \eqref{eq:FM curv} implies that 
$* (F^2)$ is the linear combination of $dx^i \wedge dy^{ab}$'s, and hence, 
$\varphi \wedge * (F^2)$ is the linear combination of $dx^{ij} \wedge dy^{4567}$'s. 
Then, we have 
$
\la \varphi \wedge * (F^2), * dy^a \ra = 0
$
for any $4 \leq a \leq 7$. Hence, we obtain 
\begin{equation} \label{eq:FM I2}
I_2 = \frac{1}{2} \varphi \wedge * (F^2). 
\end{equation}
Then, by \eqref{eq:FM asso cond}, \eqref{eq:FM I1}, \eqref{eq:FM I3} and \eqref{eq:FM I2}, 
the proof is completed. 
\end{proof}

Before going further, we rewrite the associator equality 
\cite[Chapter I\hspace{-.1em}V,Theorem 1.6]{HL}. 
This is very useful because 
Lemma \ref{lem:associator HL} implies an identity that will hold in more general settings. 
In \cite{KYvol}, we show that it indeed holds generally. 
Using this, we see that dDT connections for $G_2$-manifolds 
minimize a kind of the volume functional, which is called the Dirac-Born-Infeld (DBI) action in physics, 
and this gives further applications. For more details, see \cite{KYvol}.

\begin{lemma} \label{lem:associator HL}
We have 
\begin{align*}
&\left( 1+ \frac{1}{2} \la (F^S_\n)^2, * \varphi \ra \right)^2
+
\left| * \varphi \wedge F^S_\n + \frac{1}{6} (F^S_\n)^3 \right|^2
+
\frac{1}{4} |\varphi \wedge * (F^S_\n)^2|^2 \\
=&
\det \left( \id_{TX} + (-\i F^S_\n)^\sharp \right), 
\end{align*}
where 
$(-\i F^S_\n)^\sharp$ is a skew symmetric endomorphism of $TX$ defined by 
$$
\la (-\i F^S_\n)^\sharp u, v \ra = -\i F^S_\n (u,v) \qquad \mbox{for } u,v \in TX. 
$$
\end{lemma}
\begin{proof}
Define $\iota: B \rightarrow X=B \times T^4$ by 
$\iota (x) = (x, f(x))$. 
Set $v_i = \iota_* (\p_i)$ for $i=1,2,3$. 
Then, by the associator equality 
\cite[Chapter I\hspace{-.1em}V,Theorem 1.6]{HL}, we have 
\begin{align} \label{eq:associator HL}
|\iota^* \varphi (\p_1,\p_2,\p_3)|^2 + |*\varphi (v_1,v_2,v_3, \cdot)|^2 
= 
|v_1 \wedge v_2 \wedge v_3|^2. 
\end{align}
Then, 
since $\iota^* dx^i = dx^i$ and $\iota^* dy^a = df^a$, 
\eqref{eq:FM asso} implies that 
\begin{align*}
&\iota^* \varphi \\
=& dx^{123} 
+ dx^1 \wedge (df^{45} + df^{67})
+ dx^2 \wedge (df^{46} + df^{75})
- dx^3 \wedge (df^{47} + df^{56}) \\
=&
\left( 1+ \la dx^{23}, df^{45} + df^{67} \ra 
+ \la dx^{31}, df^{46} + df^{75} \ra - \la dx^{12}, df^{47} + df^{56} \ra \right) dx^{123},  
\end{align*}
where $df^{a b}$ is short for $df^a \wedge df^b$. 
On the other hand, since $F_\n^S = \i \sum_{a=4}^7 df^a \wedge dy^a$
by \eqref{eq:FM curv}, 
we have
\begin{align*}
\la (F^S_\n)^2, * \varphi \ra 
=&
\sum_{a,b=4}^7 \sum_{k \in \Z/3} \la df^{ab} \wedge dy^{ab}, dx^{k ,k+1} \wedge \om_{k+2} \ra \\
=&
2 \left( \la dx^{23}, df^{45} + df^{67} \ra 
+ \la dx^{31}, df^{46} + df^{75} \ra - \la dx^{12}, df^{47} + df^{56} \ra \right).  
\end{align*}
Hence, we obtain 
\begin{align} \label{eq:associator HL1}
\iota^* \varphi (\p_1,\p_2,\p_3)
= \varphi (v_1,v_2,v_3) 
=1+ \frac{1}{2} \la (F^S_\n)^2, * \varphi \ra. 
\end{align}
By the proof of Proposition \ref{prop:FM G2}, we have 
\begin{align} \label{eq:associator HL2}
\begin{split}
|*\varphi (v_1,v_2,v_3, \cdot)|^2
=&
|v_1^\flat \wedge v_2^\flat \wedge v_3^\flat \wedge \varphi|^2 \\
=&
|I_1 +I_3|^2 + |I_2|^2 \\
=&
\left| * \varphi \wedge F^S_\n + \frac{1}{6} (F^S_\n)^3 \right|^2
+
\frac{1}{4} |\varphi \wedge * (F^S_\n)^2|^2. 
\end{split}
\end{align}
Next, we compute $|v_1 \wedge v_2 \wedge v_3|^2$. 
Since $v_i=\iota_* (\p_i) = \p_i + \p f/\p x^i$, we have 
$$
|v_1 \wedge v_2 \wedge v_3|^2 
= \det \left(\id_3 + {}^t\! A A \right), 
$$
where $\id_3$ is the identity matrix of dimension 3, 
$A$ is a $4 \times 3$ matrix defined by 
$A = \left( \frac{\p f^a}{\p x^i} \right)_{4 \leq a \leq 7, 1 \leq i \leq 3}$ 
and ${}^t\! A$ is the transpose of $A$. 
Denote by 
$$
\{ 0, \pm \i \mu_1, \pm \i \mu_2, \pm \i \mu_3 \} 
\quad \mbox{and} \quad \{ \lambda_1, \lambda_2, \lambda_3 \} 
$$ 
the eigenvalues of $(-\i F^S_\n)^\sharp$ and 
${}^t\! A A$, respectively, where 
$\mu_i \in \R$ and $\lambda_j \geq 0$. 
Since 
$$
(-\i F^S_\n)^\sharp = 
\left(
\begin{array}{cc}
0 & -{}^t\! A \\
A& 0 
\end{array}
\right),  
\qquad
((-\i F^S_\n)^\sharp)^2 = 
\left(
\begin{array}{cc}
-{}^t\! A A & 0 \\
0              & -A {}^t\! A 
\end{array}
\right) 
$$
and $\{0, \lambda_1, \lambda_2, \lambda_3 \}$ are the eigenvalues of $A {}^t\! A$, 
we see that 
$$
\{ 0, \mu_1^2, \mu_2^2, \mu_3^2 \} = \{0, \lambda_1, \lambda_2, \lambda_3 \}. 
$$
Since 
$(-\i F^S_\n)^\sharp$ and $A {}^t\! A$ 
are conjugate to 
$$
0 \oplus \left(
\begin{array}{cc}
0         & -\mu_1 \\
\mu_1  & 0 
\end{array}
\right) 
\oplus
\left(
\begin{array}{cc}
0         & -\mu_2 \\
\mu_2  & 0 
\end{array}
\right) 
\oplus\left(
\begin{array}{cc}
0         & -\mu_3 \\
\mu_3  & 0 
\end{array}
\right) 
\ \mbox{and} \ 
\left(
\begin{array}{ccc}
\lambda_1& & \\
&\lambda_2& \\
& &\lambda_3 \\
\end{array}
\right), 
$$
respectively, we obtain 
\begin{align*}
\det \left( \id_{TX} + (-\i F^S_\n)^\sharp \right) 
=&
(1+\mu_1^2) (1+\mu_2^2) (1+\mu_3^2) \\
=&
(1+\lambda_1) (1+\lambda_2) (1+\lambda_3) \\
=& 
\det \left(\id_3 + {}^t\! A A \right) 
=
|v_1 \wedge v_2 \wedge v_3|^2 
\end{align*}
and the proof is completed. 
\end{proof}

Using Proposition \ref{prop:FM G2}, we obtain the following. 

\begin{proposition} \label{prop:FM G2 conn}
The following conditions are equivalent. 
\begin{enumerate}
\item
The graph $S$ is an associative submanifold with an appropriate orientation 
and $\n^B$ is flat. 
\item
$F_\n^3/6 + F_\n \wedge \ast\varphi=0$. 
\item
$F_\n^3/6 + F_\n \wedge \ast\varphi=0
\quad\mbox{and}\quad \varphi\wedge \ast F_\n^{2}=0.$
\end{enumerate}
\end{proposition}

\begin{proof}
Since $(3) \Rightarrow (2)$ is obvious and the converse holds by \cite[Remark 3.3]{KY}, 
(2) and (3) are equivalent. We show the equivalence of (1) and (2). 
By \eqref{eq:FM curv}, we have 
$(F_\n^B)^2=0$ and $F_\n^B \wedge (F_\n^S)^2=0$. Thus, we have 
$F_\n^3 = (F_\n^S)^3$ and 
$$
F_\n^3/6 + F_\n \wedge \ast\varphi 
= \left( (F_\n^S)^3/6 + F_\n^S \wedge \ast\varphi \right) + F_\n^B \wedge \ast\varphi. 
$$
By (\ref{eq:FM curv}), 
$(F_\n^S)^3/6 + F_\n^S \wedge \ast\varphi$ 
and $F_\n^B \wedge \ast\varphi$ 
are linear combinations of $dx^{123} \wedge dy^{abc}$'s  and 
$dx^{ij} \wedge dy^{4567}$'s, respectively. 
Hence, (2) is equivalent to 
$$
(F_\n^S)^3/6 + F_\n^S \wedge \ast\varphi=0 \quad\mbox{and}\quad F_\n^B \wedge \ast\varphi =0.
$$
The first equation is equivalent to saying that 
$S$ is an associative submanifold with an appropriate orientation by Proposition \ref{prop:FM G2}. 
By \eqref{eq:FM coasso} and (\ref{eq:FM curv}), 
we have 
$F_\n^B \wedge \ast\varphi = F_\n^B \wedge dy^{4567}$. 
Hence, 
$F_\n^B \wedge \ast\varphi=0$ if and only if $F_\n^B=0$. 
Then, the proof is completed. 
\end{proof}
%%%%%%%%%%%%%%%%%%%%%%%%%%%%%%%%%%%%%%%%%%%%%%%%%%%%%%%%%
%%%%%%%%%%%%%%%%%%%%%%%%%%%%%%%%%%%%%%%%%%%%%%%%%%%%%%%%%
%%%%%%%%%%%%%%%%%%%%%%%%%%%%%%%%%%%%%%%%%%%%%%%%%%%%%%%%%
%%%%%%%%%%%%%%%%%%%%%%%%%%%%%%%%%%%%%%%%%%%%%%%%%%%%%%%%%
%%%%%%%%%%%%%%%%%%%%%%%%%%%%%%%%%%%%%%%%%%%%%%%%%%%%%%%%%
%%%%%%%%%%%%%%%%%%%%%%%%%%%%%%%%%%%%%%%%%%%%%%%%%%%%%%%%%
\section{The real Fourier--Mukai transform for Cayley $T^4$-fibrations} \label{sec:dDT Spin7}
In this section, we compute the real Fourier--Mukai transform of Cayley cycles and prove main theorems. 

Let $B\subset\mathbb{R}^4$ be an open set with coordinates $(x^0, x^{1},x^{2},x^{3})$ and 
$f=(f^{4},f^{5},f^{6},f^{7}):B\to T^{4}$ be a smooth function with values in $T^{4}$. 
We use coordinates $(y^{4},y^{5},y^{6},y^{7})$ for $T^{4}$. 
Put
\[S:=\{\,(x,f(x))\mid x\in B\,\}\]
the graph of $f$, a 4-dimensional submanifold in $X:=B\times T^{4}$. 
The standard ${\rm Spin}(7)$-structure $\Phi$ on $X$ is described as 
\[
\Phi = dx^0 \wedge \varphi + *_7 \varphi, 
\]
where we use $\varphi$ in \eqref{eq:FM asso} and $*_7$ is the Hodge star operator on 
$( \{\, 0 \,\} \times \R^3) \times T^4$. 
Setting 
\[
\begin{aligned}
\tau_1 &=dx^{01} + dx^{23}, &  \tau_2&=dx^{02} + dx^{31}, & \tau_3 &=dx^{03} + dx^{12}, \\
\om_1 &=dy^{45} + dy^{67}, & \om_2 &=dy^{46} + dy^{75}, & \om_3 &=-(dy^{47} + dy^{56}). 
\end{aligned}
\]
$\Phi$ is also described as 
\begin{equation}\label{eq:FM Cayley}
\Phi = dx^{0123} + dy^{4567} + \sum_{i=1}^3 \tau_i \wedge \om_i.  
\end{equation}
Note that $\Phi$ in \eqref{Phi4} is also described as in \eqref{eq:FM Cayley}. 
Let 
\[
\n^B = d + \i \sum_{j=0}^3 A^j dx^j
\]
be a Hermitian connection of a trivial complex line bundle $B \times \C \to B$, 
where 
$A^j : B \to \R$ is a smooth function.

Next, we consider the mirror side. 
The real Fourier--Mukai transform of $(S, \n^B)$ is the connection on $X^{*}(\cong X)$ defined by 
\[
\nabla:=d+ \i \sum_{j=0}^3 A^{j} dx^j + \i \sum_{a=4}^7 f^{a} dy^{a}. 
\]
Then, its curvature 2-form $F_\nabla$ is described as $F_\nabla = F_\n^B + F_\n^S$, where 
\begin{equation} \label{eq:FM curv Spin7 0}
F_\n^B 
= \i \sum_{i, j=0}^3 \frac{\partial A^{j}}{\partial x^{i}} dx^i \wedge dx^j, \quad 
F_\n^S
= \i \sum_{i=0}^3 \sum_{a=4}^7 \frac{\partial f^{a}}{\partial x^{i}} dx^i \wedge dy^a. 
\end{equation}
Note that 
the real Fourier--Mukai transform of $S$ 
is the connection on $X^{*}(\cong X)$ defined by 
\[
\nabla^S:=d + \i \sum_{a=4}^7 f^{a} dy^{a} 
\]
and its curvature 2-form is given by $F_\n^S$. 
We first describe the condition for $S$ to be Cayley in terms of $F_\n^S$ in Theorem \ref{thm:FM Spin7}. 
Using this, we show that the similar statement also holds for $F_\n$ in Theorem \ref{thm:FM Spin7 conn}. 

\begin{theorem}\label{thm:FM Spin7}
Use the notation of Subsection \ref{sec:Spin7 geometry}. 
The graph $S$ is a Cayley submanifold with an appropriate orientation 
if and only if 
\[
\pi^2_7 \left( F_\nabla^S + \frac{1}{6} *_8 (F_\nabla^S)^3 \right) =0
\quad\mbox{and}\quad  \pi^4_7  \left((F_\nabla^S)^{2} \right)=0.
\]
\end{theorem}

\begin{remark} \label{rem:FM Spin7}
A similar statement for graphical submanifolds 
is given by Harvey and Lawson in  \cite[Chapter I\hspace{-.1em}V,Theorem 2.20]{HL}. 
They showed that 
$S$ is a Cayley submanifold with an appropriate orientation 
if and only if 
two equations (1)' and (2)' are satisfied. 
These equations are given in terms of differential equations for $f$ 
and correspond to the two equations above. 
They also showed that 
if the determinant of the Jacobian of $f$ is never 1, (1)' implies (2)'. 
This is generalized in \cite{KYSpin7}. 

Thus, unlike the $G_2$ case (\cite[Remark 3.3]{KY}), 
the first equation \emph{does not} always imply the second. 
Counterexamples are provided in \cite[p.132]{HL}.  
\end{remark}

\begin{proof}
Set 
\[
F: = - \i F_\nabla^S = d x^0 \wedge F_1 + F_2
\]
where 
$F_1: B \to (\R^7)^*$ and $F_2: B \to \Lambda^2 (\R^7)^*$ are given by 
\begin{equation} \label{eq:FM curv Spin7}
F_1 = \sum_{a=4}^7 \frac{\partial f^{a}}{\partial x^{0}} dy^{a}, \quad 
F_2 = \sum_{i=1}^3 \sum_{a=4}^7 \frac{\partial f^{a}}{\partial x^{i}} dx^i \wedge dy^a. 
\end{equation}
Set 
$\partial_{i}:=\partial/\partial x^i$ and $\partial_a:=\partial/\partial y^a$ 
for $0 \leq i \leq 3$ and $4 \leq a \leq 7$. 
By Lemma \ref{lem:Cayley char}
and equations \eqref{eq:Spin7 proj 1} and \eqref{eq:def tau},  
$S$ is Cayley with an appropriate orientation if and only if  
$\lambda^4 (\alpha) (v_0, v_{1},v_{2},v_{3})=0$ for any 
$p\in S$, $v_0, v_{1},v_{2},v_{3}\in T_{p}S$ and 
$\alpha \in \rmspan \{\, dx^1, \cdots, dx^3, dy^4, \cdots, dy^7 \,\}$. 
The tangent space of $S$ is spanned by $v_0, v_1,v_2,v_3$, where 
\[
v_{j}:=\p_j + \sum_{a=4}^7 \frac{\partial f^{a}}{\partial x^{j}} \p_a 
\]
for $0 \leq j \leq 3$. By \eqref{eq:FM curv Spin7}, note that 
\[
v_0 = \p_0 + F_1^\sharp, \quad 
v_j = \p_j + (i(\p_j) F_2)^\sharp 
\]
for $1 \leq j \leq 3$. 
Then, we compute 
\[
\begin{aligned}
&\sqrt{8} \lambda^4 (\alpha) (v_0, v_{1},v_{2},v_{3}) \\
=& 
\left( d x^0 \wedge i (\alpha^\sharp) \ast_7 \varphi - \alpha \wedge \varphi \right) (v_0, v_{1},v_{2},v_{3}) \\
=&
*_7 \varphi (\alpha^\sharp, v_{1},v_{2},v_{3}) 
- \alpha (v_0) \varphi (v_{1},v_{2},v_{3}) + \sum_{k \in \Z/3} \alpha (v_k) \varphi (v_0, v_{k+1}, v_{k+2})\\
=&
- \la *_7 \varphi (v_{1},v_{2},v_{3}, \,\cdot\,), \alpha \ra 
- \la \alpha, F_1 \ra \varphi (v_{1},v_{2},v_{3}) \\
&+ \sum_{k \in \Z/3} \alpha (v_k) \varphi (F_1^\sharp, v_{k+1}, v_{k+2}). 
\end{aligned}
\]
Since 
$- *_7 \varphi (v_{1},v_{2},v_{3}, \,\cdot\,) = 
- i(v_3) i(v_2) i(v_1) *_7 \varphi 
= - *_7 (v_3^\flat \wedge v_2^\flat \wedge v_1^\flat \wedge \varphi)
= *_7 (v_1^\flat \wedge v_2^\flat \wedge v_3^\flat \wedge \varphi), 
$
we have 
\begin{align*}
&\sqrt{8} \lambda^4 (\alpha) (v_0, v_{1},v_{2},v_{3}) \\
=&
\left \la v_1^\flat \wedge v_2^\flat \wedge v_3^\flat \wedge \varphi 
- \varphi (v_{1},v_{2},v_{3}) *_7 F_1 
+ \sum_{k \in \Z/3} \varphi (F_1^\sharp, v_{k+1}, v_{k+2}) *_7 v_k^\flat, *_7 \alpha \right \ra. 
\end{align*}
By the proof of Proposition \ref{prop:FM G2}, we have 
\[
v_1^\flat \wedge v_2^\flat \wedge v_3^\flat \wedge \varphi 
= I_1 + I_2 + I_3, 
\]
where 
\begin{equation} \label{eq:Spin7 I123}
I_1 = -*_7 \varphi \wedge F_2, \quad I_2 = \frac{1}{2} \varphi \wedge *_7 F_2^2, 
\quad I_3 = \frac{1}{6} F_2^3. 
\end{equation}
Here, we set 
\begin{align}
\label{eq:Cayley 1}
\begin{split}
J_1 &= I_1 + I_3 - \varphi (v_{1},v_{2},v_{3}) *_7 F_1 
+ \sum_{k \in \Z/3} \varphi (F_1^\sharp, v_{k+1}, v_{k+2}) *_7 (i(\p_k) F_2), \\
J_2&= I_2 + \sum_{k \in \Z/3} \varphi (F_1^\sharp, v_{k+1}, v_{k+2}) *_7 dx^k. 
\end{split}
\end{align}
Then, 
\begin{align} \label{eq:Cayley lam4}
\sqrt{8} \lambda^4 (\alpha) (v_0, v_{1},v_{2},v_{3}) 
= \la J_1 + J_2, *_7 \alpha \ra
\end{align}
Since $*_7 I_1, *_7 I_3, F_1$ are linear combinations of $dy^a$'s 
and $*_7 I_2$ is a linear combination of $dx^i$'s, 
the graph $S$ is Cayley with an appropriate orientation if and only if
\begin{align*}
J_1=0 \quad \mbox{and} \quad J_2=0.  
\end{align*}
To simplify these equations, we show the following. 

\begin{lemma} \label{lem:Cayley simplify}
We have 
\[
\begin{aligned}
\varphi (v_{1},v_{2},v_{3}) &= 1- \frac{1}{2} *_7 \left( \varphi \wedge F_2^2 \right), \\
\sum_{k \in \Z/3} \varphi (F_1^\sharp, v_{k+1}, v_{k+2}) d x^k&
= - *_7 (F_1 \wedge F_2 \wedge \varphi). 
\end{aligned}
\]
\end{lemma}
\begin{proof}
The first equation follows from \eqref{eq:associator HL1}. 
We prove the second equation. 
Since $F_1$ is a linear combination of $dy^a$'s, 
the equation \eqref{eq:FM asso} implies that 
\[
\begin{aligned}
\varphi (F_1^\sharp, v_{k+1}, v_{k+2})
=
\varphi (F_1^\sharp, \p_{k+1}, (i(\p_{k+2}) F_2)^\sharp) 
+
\varphi (F_1^\sharp, (i(\p_{k+1}) F_2)^\sharp, \p_{k+2}). 
\end{aligned}
\]
We compute 
\[
\begin{aligned}
\varphi (F_1^\sharp, \p_{k+1}, (i(\p_{k+2}) F_2)^\sharp) 
&=
- \om_{k+1} (F_1^\sharp, (i(\p_{k+2}) F_2)^\sharp) \\
&=
- \la F_1 \wedge i(\p_{k+2}) F_2, \om_{k+1} \ra\\
&=
\la i(\p_{k+2}) (F_1 \wedge F_2), \om_{k+1} \ra\\
&=
\la F_1 \wedge F_2, d x^{k+2} \wedge \om_{k+1} \ra. 
\end{aligned}
\]
Then, we have 
\[
\begin{aligned}
\varphi (F_1^\sharp, v_{k+1}, v_{k+2})
&= 
\la F_1 \wedge F_2, d x^{k+2} \wedge \om_{k+1} - d x^{k+1} \wedge \om_{k+2} \ra \\
&=
- \la F_1 \wedge F_2, i(\p_{k}) *_7 \varphi \ra \\
&=
- \la d x^k \wedge F_1 \wedge F_2, *_7 \varphi \ra
=
- \la d x^k, *_7 (F_1 \wedge F_2 \wedge \varphi) \ra. 
\end{aligned}
\]
By \eqref{eq:FM asso} and \eqref{eq:FM curv Spin7}, 
$F_1 \wedge F_2 \wedge \varphi$ is a linear combination of $d x^{i j} \wedge d y^{4567}$'s, 
and hence, the proof is completed. 
\end{proof}

Thus, by \eqref{eq:Spin7 I123}, \eqref{eq:Cayley 1} and Lemma \ref{lem:Cayley simplify}, 
we see that 
\begin{align}
\label{eq:Cayley 2}
\begin{split}
J_1 =& - *_7 \varphi \wedge F_2 + \frac{1}{6} F_2^3 
-\left( 1-\frac{1}{2} *_7 \left( \varphi \wedge F_2^2 \right) \right) *_7 F_1 
\\
&+ *_7 (F_1 \wedge F_2 \wedge \varphi) \wedge *_7 F_2 , 
\\  
J_2 =& 
\frac{1}{2} \varphi \wedge *_7 F_2^2 
- F_1 \wedge F_2 \wedge \varphi .
\end{split}
\end{align}
Now, we describe 
$\pi^2_7 (F - *_8 F^3/6)$ and $\pi^4_7 (F^2)$. 

\begin{lemma} \label{lem:2747}
We have 
\begin{equation*}
\begin{aligned}
2 \pi^2_7 \left(F - \frac{1}{6} *_8 F^3 \right) 
=&
\lambda^2 \left(
*_7 \left( *_7 \varphi \wedge F_2 - \frac{1}{6} F_2^3 
+\left( 1-\frac{1}{2} *_7 \left( \varphi \wedge F_2^2 \right) \right) *_7 F_1 \right. \right.\\
&\left. \left. - *_7 (F_1 \wedge F_2 \wedge \varphi) \wedge *_7 F_2
\right) \right), \\
\sqrt{8} \pi^4_7 (F^2) =& \lambda^4 
\left(*_7 \left( 2 F_1 \wedge F_2 \wedge \varphi - \varphi \wedge *_7 F_2^2 \right) \right). 
\end{aligned}
\end{equation*}
\end{lemma}

\begin{proof}
Set 
$$
\{\, e^0, \cdots, e^7 \,\} = \{\, dx^0, \cdots, dx^3, dy^4, \cdots, dy^7 \,\} \quad \mbox{and} \quad 
\{\, e_0, \cdots, e_7 \,\} = \{\, \p_0, \cdots, \p_7 \,\}.
$$ 
Then, by \eqref{eq:Spin7 proj 1} and \eqref{def:lambda}, we have 
\[
\begin{aligned}
2 \pi^2_7 (F)
=&  2 \sum_{\mu=1}^7 \la F, \lambda^2 (e^\mu) \ra \cdot \lambda^2 (e^\mu)\\
=&  \sum_{\mu=1}^7 \la e^0 \wedge F_1 + F_2, e^0 \wedge e^\mu + i (e_\mu) \varphi  \ra \cdot \lambda^2 (e^\mu)\\
=& \sum_{\mu=1}^7 \left( \la F_1, e^\mu \ra  + \la F_2, i (e_\mu) \varphi  \ra \right) \cdot \lambda^2 (e^\mu). 
\end{aligned}
\]
Since 
$\la F_2, i (e_\mu) \varphi  \ra 
= *_7 (F_2 \wedge e^\mu \wedge *_7 \varphi) 
= \la e^\mu, *_7 (F_2 \wedge *_7 \varphi) \ra, 
$
we obtain 
\begin{equation} \label{eq:*F proj}
2 \pi^2_7 (F) = \lambda^2 \left(F_1 + *_7 (F_2 \wedge *_7 \varphi) \right). 
\end{equation}
We also compute 
\[
\pi^2_7 (*_8 F^3)
= \sum_{\mu=1}^7 \la *_8 F^3, \lambda^2 (e^\mu) \ra \cdot \lambda^2 (e^\mu) 
= \sum_{\mu=1}^7 \la F^3, \lambda^6 (e^\mu) \ra \cdot \lambda^2 (e^\mu). 
\]
By \eqref{eq:Hodge 78}, we have for $1 \leq \mu \leq 7$  
\[
2 \lambda^6 (e^\mu) 
= *_8 \left( e^0 \wedge e^\mu + i(e_\mu) \varphi \right)
= *_7 e^\mu + e^\mu \wedge *_8 \varphi
= *_7 e^\mu + e^0 \wedge e^\mu \wedge *_7 \varphi, 
\]
and hence, 
\[
\begin{aligned}
2 \la F^3, \lambda^6 (e^\mu) \ra
&= 
\la 3 e^0 \wedge F_1 \wedge F_2^2 + F_2^3, *_7 e^\mu 
+ e^0 \wedge e^\mu \wedge *_7 \varphi \ra\\
&=
3 \la F_1 \wedge F_2^2, e^\mu \wedge *_7 \varphi \ra 
+ \la F_2^3, *_7 e^\mu \ra. 
\end{aligned}
\]
The first term is computed as 
\[
\begin{aligned}
3 \la F_1 \wedge F_2^2, e^\mu \wedge *_7 \varphi \ra 
=& 
3 \la i (e_\mu) \left( F_1 \wedge F_2^2 \right), *_7 \varphi \ra \\
=&
3 *_7 \left(  \left(  \la e^\mu, F_1 \ra F_2^2 
- 2 F_1 \wedge ( i (e_\mu) F_2) \wedge F_2 \right) \wedge \varphi \right)\\
=& 
3 \la *_7 (F_2^2 \wedge \varphi ) F_1, e^\mu \ra 
-6 *_7 \left( F_1 \wedge ( i (e_\mu) F_2) \wedge F_2 \wedge \varphi \right). 
\end{aligned}
\]
The second term is computed as 
\[
\begin{aligned}
- 6 *_7 \left( F_1 \wedge ( i (e_\mu) F_2) \wedge F_2 \wedge \varphi \right)
&= 6 \la  i (e_\mu) F_2, *_7 \left( F_1 \wedge F_2 \wedge \varphi \right) \ra \\
&= 
- 6 *_7 \left( *_7 \left( F_1 \wedge F_2 \wedge \varphi \right) \wedge e^\mu \wedge *_7 F_2 \right) \\
&=
6 \la *_7 \left( *_7 \left( F_1 \wedge F_2 \wedge \varphi \right) \wedge *_7 F_2 \right), e^\mu \ra. 
\end{aligned}
\]
Summarizing these equations, we obtain 
\begin{equation} \label{eq:F3 proj}
2 \pi^2_7 (*_8 F^3)
= \lambda^2 
\left( *_7 \left( F_2^3 + 3 *_7 (F_2^2 \wedge \varphi ) *_7 F_1
+ 6 *_7 \left( F_1 \wedge F_2 \wedge \varphi \right) \wedge *_7 F_2\right) \right). 
\end{equation}
Then, by \eqref{eq:*F proj} and \eqref{eq:F3 proj}, it follows that 
\begin{equation*}
\begin{aligned}
&2 \pi^2_7 \left(F - \frac{1}{6} *_8 F^3 \right) \\
=& \lambda^2 
\left( F_1 + *_7 (F_2 \wedge *_7 \varphi) \right. \\
& \left. - \frac{1}{6} 
*_7 \left( F_2^3 + 3 *_7 (F_2^2 \wedge \varphi ) *_7 F_1 
+ 6 *_7 \left( F_1 \wedge F_2 \wedge \varphi \right) \wedge *_7 F_2 \right) \right) \\
=&
\lambda^2 \left(
*_7 \left( *_7 \varphi \wedge F_2 - \frac{1}{6} F_2^3 
+\left( 1-\frac{1}{2} *_7 \left( \varphi \wedge F_2^2 \right) \right) *_7 F_1 \right. \right.\\
&\left. \left. - *_7 (F_1 \wedge F_2 \wedge \varphi) \wedge *_7 F_2
\right) \right), 
\end{aligned}
\end{equation*}
which implies the first equation of Lemma \ref{lem:2747}.

Next, we compute $\pi^4_7 (F^2)$. 
By \eqref{eq:Spin7 proj 1}, we have 
$$
\pi^4_7 (F^2)
= \sum_{\mu=1}^7 \la F^2, \lambda^4 (e^\mu) \ra \cdot \lambda^4 (e^\mu). 
$$ 
For $1 \leq \mu \leq 7$, we have by \eqref{def:lambda}
\[
\begin{aligned}
\sqrt{8} \la F^2, \lambda^4 (e^\mu) \ra 
=&   
\la 2 e^0 \wedge F_1 \wedge F_2 + F_2^2, 
e^0 \wedge i (e_\mu) \ast_7 \varphi - e^\mu \wedge \varphi \ra \\
=&
2 \la F_1 \wedge F_2, i (e_\mu) \ast_7 \varphi \ra - \la F_2^2, e^\mu \wedge \varphi \ra
\end{aligned}
\]
and 
\[
\begin{aligned}
2 \la F_1 \wedge F_2, i (e_\mu) \ast_7 \varphi \ra
=&
- 2 *_7 \left( F_1 \wedge F_2 \wedge e^\mu \wedge \varphi \right) \\
=& 
2 \la *_7 \left( F_1 \wedge F_2 \wedge \varphi \right), e^\mu \ra, \\
- \la F_2^2, e^\mu \wedge \varphi \ra
=&
- *_7 (e^\mu \wedge \varphi \wedge *_7 F_2^2)
=
- \la *_7 \left( \varphi \wedge *_7 F_2^2 \right), e^\mu \ra. 
\end{aligned}
\]
Hence, we obtain the second equation of Lemma \ref{lem:2747}. 
\end{proof}

Then, by \eqref{eq:Cayley 2} and Lemma \ref{lem:2747}, we obtain 
\begin{align}
*_7 J_1 &= 2 (\lambda^2)^{-1} \left( \pi^2_7 \left( -F + \frac{1}{6} *_8 F^3 \right) \right), \label{eq:CayleyJ1} \\  
*_7 J_2 &= - \frac{\sqrt{8}}{2} (\lambda^4)^{-1} \pi^4_7 (F^2). \label{eq:CayleyJ2}
\end{align}
Hence, by \eqref{eq:CayleyJ1} and \eqref{eq:CayleyJ2}, 
we see that 
the graph $S$ is Cayley with an appropriate orientation if and only if
$\pi^2_7 \left( F - *_8 F^3/6 \right) = 0$ and $\pi^4_7 (F^2) = 0$. 
\end{proof}

Before going further, we rewrite the Cayley equality 
\cite[Chapter I\hspace{-.1em}V,Theorem 1.28]{HL}. 
This is very useful because 
Lemma \ref{lem:Cayley HL} implies an identity that will hold in more general settings as in Lemma \ref{lem:associator HL}. 
We show that it indeed holds generally and gives many applications. For more details, see \cite{KYvol}.

\begin{lemma} \label{lem:Cayley HL}
We have 
\begin{align*}
&\left( 1+ \frac{1}{2} \la (F_\n^S)^2, \Phi \ra + \frac{*_8 (F_\n^S)^4}{24}  \right)^2 
+
4 \left| \pi^2_7 \left( F_\n^S + \frac{1}{6} *_8 (F_\n^S)^3\right) \right|^2
+
2 \left| \pi^4_7 \left( (F_\n^S)^2 \right) \right|^2 \\
=&
\det (\id_{TX} + (-\i F_\n^S)^\sharp), 
\end{align*}
where 
$(-\i F^S_\n)^\sharp$ is a skew symmetric endomorphism of $TX$
defined by 
$$
\la (-\i F^S_\n)^\sharp u, v \ra = -\i F^S_\n (u,v) \qquad \mbox{for } u,v \in TX. 
$$
\end{lemma}
\begin{proof}
Define $\iota: B \rightarrow X=B \times T^4$ by 
$\iota (x) = (x, f(x))$. 
Set $v_i = \iota_* (\p_i)$ for $i=0,1,2,3$. 
Then, by the Cayley equality 
\cite[Chapter I\hspace{-.1em}V,Theorem 1.6]{HL}, 
which is equivalent to Lemma \ref{lem:Cayley id}, 
we have 
\begin{align} \label{eq:Cayley HL}
|\iota^* \Phi (\p_0, \p_1,\p_2,\p_3)|^2 + 8 |\tau (v_0, v_1,v_2,v_3)|^2 
= 
|v_0 \wedge v_1 \wedge v_2 \wedge v_3|^2, 
\end{align}
where $\tau$ is defined by \eqref{eq:def tau}. 
Then, 
since $\iota^* dx^i = dx^i$ and $\iota^* dy^a = df^a$, 
\eqref{eq:FM Cayley} implies that 
\begin{align*}
&\iota^* \Phi \\
=& dx^{0123} + df^{4567} 
+ \tau_1 \wedge (df^{45} + df^{67})
+ \tau_2 \wedge (df^{46} + df^{75})
- \tau_3 \wedge (df^{47} + df^{56}) \\
=&
\left( 1+ *_4 (df^{4567}) + \la \tau_1, df^{45} + df^{67} \ra 
+ \la \tau_2, df^{46} + df^{75} \ra - \la \tau_3, df^{47} + df^{56} \ra \right) dx^{0123}, 
\end{align*}
where $*_4$ is the Hodge star on the space spanned by $dy^4, \cdots, dy^7$ 
and $df^{a_1 \cdots a_k}$ is short for $df^{a_1} \wedge \cdots \wedge df^{a_k}$.
On the other hand, since $F_\n^S = \i \sum_{a=4}^7 df^a \wedge dy^a$
by \eqref{eq:FM curv Spin7 0}, 
we have
\begin{align*}
\frac{*_8 (F_\n^S)^4}{24}
=&
\sum_{a,b,c,d=4}^7 *_8 \left(\frac{df^{abcd} \wedge dy^{abcd}}{24} \right)
=
*_8 \left(df^{4567} \wedge dy^{4567} \right)
=
*_4 \left(df^{4567} \right), \\
\la (F^S_\n)^2, \Phi \ra 
=&
\sum_{a,b=4}^7 \sum_{i=1}^3 \la df^{ab} \wedge dy^{ab}, \tau_i \wedge \om_i \ra \\
=&
\sum_{a,b=4}^7 \sum_{i=1}^3 \la df^{ab}, \tau_i \ra \la dy^{ab}, \om_i \ra \\
=&
2 \left( \la \tau_1, df^{45} + df^{67} \ra 
+ \la \tau_2, df^{46} + df^{75} \ra - \la \tau_3, df^{47} + df^{56} \ra \right).  
\end{align*}
Hence, we obtain 
\begin{align} \label{eq:Cayley HL1}
|\iota^* \Phi (\p_0, \p_1,\p_2,\p_3)|^2
=
\left( 1+ \frac{1}{2} \la (F_\n^S)^2, \Phi \ra + \frac{*_8 (F_\n^S)^4}{24}  \right)^2. 
\end{align}

Next, we compute $8 |\tau (v_0, v_1,v_2,v_3)|^2$. 
By \eqref{eq:def tau} and \eqref{eq:Cayley lam4}, 
we have 
\begin{align*} 
8 |\tau (v_0, v_1,v_2,v_3)|^2=
8 \sum_{\mu=1}^7 \left( \lambda^4(e^\mu) (v_0, v_1, v_2, v_3) \right)^2 
=
\sum_{\mu=1}^7 \left \la *_7 J_1 + *_7 J_2, e^\mu \right \ra^2. 
\end{align*}
Recall that $*_7 J_1$ and $*_7 J_2$ are linear combinations of $dy^a$'s 
and $dx^i$'s, respectively, 
and $\lambda^j$ is an isometry by Lemma \ref{lem:lambdas}. 
Then, by \eqref{eq:CayleyJ1} and \eqref{eq:CayleyJ2}, we obtain 
\begin{align} \label{eq:Cayley HL2}
\begin{split}
8 |\tau (v_0, v_1,v_2,v_3)|^2
=&
\sum_{\mu=1}^7 
\left \la *_7 J_1, e^\mu \right \ra^2 + \left \la *_7 J_2, e^\mu \right \ra^2 \\
=&
4 \left| \pi^2_7 \left( F_\n^S + \frac{1}{6} *_8 (F_\n^S)^3\right) \right|^2
+
2 \left| \pi^4_7 \left( (F_\n^S)^2 \right) \right|^2. 
\end{split}
\end{align}
By the same argument as in the proof of Lemma \ref{lem:associator HL}, 
we see that 
$$
|v_0 \wedge v_1 \wedge v_2 \wedge v_3|^2
= \det \left( \id_{TX} + (-\i F^S_\n)^\sharp \right) 
$$ 
and the proof is completed. 
\end{proof}

Using Theorem \ref{thm:FM Spin7}, we obtain the following Theorem \ref{thm:FM Spin7 conn}. 
We first prove the following lemma.

\begin{lemma} \label{lem:Cayley ASD}
Let $U \subset \R^8$ be a Cayley subspace, 
a subspace of $\R^8$ which is a Cayley submanifold. Denote by $U^\perp$ the orthogonal complement of $U$. 
We identify $\Lambda^k U^*$ with the subspace of $\Lambda^k (\R^8)^*$ by 
\[
\Lambda^k U^* = \{\, \alpha \in \Lambda^k (\R^8)^* \mid i(v) \alpha =0 
\mbox{ for any } v \in U^\perp \,\}. 
\]
Then, 
$\alpha \in \Lambda^2 U^*$ is anti-self-dual with respect to the induced metric 
if and only if $\pi^2_7 (\alpha) = 0$. 
\end{lemma}

\begin{proof}
Since $U$ is Cayley, there is an orthonormal basis 
$$\{\, \p/\p x^0, \cdots, \p/\p x^3, \p/\p y^4, \cdots, \p/\p y^7  \,\}$$ 
with its dual 
$\{\, dx^0, \cdots, dx^3, dy^4, \cdots, dy^7  \,\}$ 
such that 
$U$ is spanned by $\p/\p x^0, \cdots, \p/\p x^3$, 
which is positively oriented, 
$U^\perp$ is spanned by $\p/\p y^4, \cdots, \p/\p y^7$
and \eqref{eq:FM Cayley} holds. 

Denote by $*_4$ and $*_8$ the Hodge stars on $U$ and $\R^8$, 
respectively. 
Then, by \eqref{eq:Spin7 proj 2}, we have 
\begin{align*}
4 \pi^2_7 (\alpha) 
=& 
\alpha + *_8 (\Phi \wedge \alpha) \\
=&
\alpha + *_8 \left( dy^{4567} \wedge \alpha + \sum_{j=1}^3 \alpha \wedge \tau_i \wedge \om_i \right) \\
=&
\alpha + *_4 \alpha 
+ *_8 \left( \sum_{j=1}^3 \la \alpha, \tau_i \ra dx^{0123} \wedge \om_i \right)\\
=&
\alpha + *_4 \alpha 
+ \sum_{j=1}^3 \la \alpha, \tau_i \ra \om_i. 
\end{align*}
Since $\{ \p/\p x^0, \cdots, \p/\p x^3 \}$ is positively oriented, 
$\{\, \tau_1, \tau_2, \tau_3 \,\}$ is a basis of the space of self-dual 2-forms on $U$. 
Hence, the proof is completed. 
\end{proof}

\begin{theorem} \label{thm:FM Spin7 conn}
The following conditions are equivalent. 
\begin{enumerate}
\item
The graph $S$ is a Cayley submanifold with an appropriate orientation 
and if we identify $- \i F_\n^B \in \Om^2(B)$ with a 2-form on $S$, 
it is anti-self-dual with respect to the induced metric 
and the orientation which makes $S$ Cayley. 
\item
\[
\pi^2_7 \left( F_\nabla + \frac{1}{6} *_8 F_\nabla^3 \right) =0
\quad\mbox{and}\quad  \pi^4_7  \left(F_\nabla^{2} \right)=0.
\]
\end{enumerate}
\end{theorem}

\begin{proof}
By \eqref{eq:FM curv Spin7 0}, we have 
$(F_\n^B)^3=0$ and $(F_\n^B)^2 \wedge F_\n^S = 0$. 
Thus, we have 
$F_\n^3 = 3 F_\n^B \wedge (F_\n^S)^2 + (F_\n^S)^3$.  
Hence, 
\begin{align*}
\pi^2_7 \left( F_\nabla + \frac{1}{6} *_8 F_\nabla^3 \right)
=
\pi^2_7 \left( \left( F^S_\nabla + \frac{1}{6} *_8 (F^S_\nabla)^3 \right) 
+ \left( F^B_\nabla + \frac{1}{2} *_8 \left(F^B_\nabla \wedge (F^S_\n)^2 \right) \right)
\right). 
\end{align*}
Note that 
$F^S_\nabla + *_8 (F^S_\nabla)^3/6$, 
$F^B_\nabla$ and $*_8 \left(F^B_\nabla \wedge (F^S_\n)^2/2 \right)$
are linear combinations of $dx^i \wedge dy^a$'s, 
$dx^{ij}$'s and $dy^{ab}$'s, respectively. 
Then, by \eqref{eq:Spin7 proj 2} and \eqref{eq:FM Cayley}, 
the first term $\pi^2_7\left(F^S_\nabla + *_8 (F^S_\nabla)^3/6 \right)$
is a linear combination of $dx^i \wedge dy^a$'s 
and the second term 
$\pi^2_7 \left(F^B_\nabla + *_8 \left( F^B_\nabla \wedge (F^S_\n)^2/2 \right) \right)$
is that of $dx^{ij}$'s and $dy^{ab}$'s. 
Hence, 
$\pi^2_7 \left( F_\nabla + *_8 F_\nabla^3/6 \right)=0$ 
if and only if 
\begin{equation} \label{eq:Cayley 5}
\pi^2_7\left(F^S_\nabla + \frac{1}{6} *_8 (F^S_\nabla)^3 \right) =0, \quad 
\pi^2_7 \left(F^B_\nabla + \frac{1}{2} *_8 \left( F^B_\nabla \wedge (F^S_\n)^2 \right) \right)=0. 
\end{equation}

Next, we consider $\pi^4_7  \left(F_\nabla^{2} \right)
= \pi^4_7  \left((F^B_\n)^2 + 2 F^B_\n \wedge F^S_\n + (F^S_\n)^2 \right)=0$. 
By the definition of $\lambda^4$ in \eqref{def:lambda}, \eqref{eq:FM asso} and \eqref{eq:FM coasso}, 
$\lambda^4(dx^i)$ is a linear combination of $dx^{jk} \wedge dy^{bc}$'s 
for each $1 \leq i \leq 3$, 
and 
$\lambda^4(dy^a)$ is a linear combination of $dx^{j} \wedge dy^{bcd}$'s 
and $dx^{j k \l} \wedge dy^{b}$'s for each $4 \leq a \leq 7$. 
By \eqref{eq:FM curv Spin7 0}, 
$(F^B_\n)^2, F^B_\n \wedge F^S_\n$ and $(F^S_\n)^2$ 
are linear combinations of 
$dx^{0123}$, $dx^{ijk} \wedge dy^a$'s  and $dx^{ij} \wedge dy^{ab}$'s,  respectively. 
Hence, 
we have $\pi^4_7 \left((F^B_\n)^2 \right)=0$ and 
$\pi^4_7  \left(F_\nabla^{2} \right)=0$ if and only if
\begin{align} \label{eq:Cayley 6}
\pi^4_7  \left((F^S_\n)^2 \right)=0, \quad 
\pi^4_7  \left(F^B_\n \wedge F^S_\n \right)=0. 
\end{align}

The first equations of \eqref{eq:Cayley 5} and \eqref{eq:Cayley 6} 
are equivalent to saying that 
$S$ is a Cayley submanifold with an appropriate orientation by Theorem \ref{thm:FM Spin7}. 
Thus, assuming that $S$ is a Cayley submanifold, 
we may show that 
$\pi^2_7 \left(F^B_\nabla + *_8 \left( F^B_\nabla \wedge (F^S_\n)^2/2 \right) \right)=0$ 
and 
$\pi^4_7  \left(F^B_\n \wedge F^S_\n \right)=0$ 
if and only if 
$-\i F_\n^B$ is anti-self-dual with respect to the induced metric 
and the orientation which makes $S$ Cayley. 
For simplicity, set 
$$
(F^S)^\sharp = (-\i F^S_\n)^\sharp, 
$$
where $(-\i F^S_\n)^\sharp$ is defined in Lemma \ref{lem:Cayley HL}. 
Then,  assuming that $S$ is a Cayley submanifold, 
$\pi^2_7 \left( F^B_\nabla + *_8 \left( F^B_\nabla \wedge (F^S_\n)^2/2 \right) \right)=0$ 
and 
$\pi^4_7  \left(F^B_\n \wedge F^S_\n \right)=0$ 
if and only if 
\begin{align}\label{eq:FM Spin7 conn 2}
\pi^2_7 \left( \left( (\id_{TX} + (F^S)^\sharp)^{-1} \right)^* F^B_\n \right) = 0
\end{align}
by \cite[Theorem A.8 (2)]{KYSpin7}.

Now, we observe \eqref{eq:FM Spin7 conn 2} pointwisely.
Fix $x \in B$ and regard $(F^S)^\sharp = (F^S)^\sharp_{(x, f(x))} \in {\rm End}(T_{(x, f(x))} X) 
\cong  {\rm End}(\R^8)$. 
By the definition of $(F^S)^\sharp$, we see that 
\[
\begin{aligned}
(\id_{TX}+(F^S)^\sharp)_* \left(\frac{\p}{\p x^i} \right) 
&= \frac{\p}{\p x^i} + \sum_{a=4}^7 \frac{\p f^a}{\p x^i} (x) \frac{\p}{\p y^a}, \\
(\id_{TX}+(F^S)^\sharp)_* \left(\frac{\p}{\p y^a} \right) 
&= \frac{\p}{\p y^a} - \sum_{i=0}^3 \frac{\p f^a}{\p x^i} (x) \frac{\p}{\p x^i}.
\end{aligned}
\]
Then, we can regard 
$(\id_{TX}+(F^S)^\sharp)_* \left(\p/\p x^i \right)$ as an element of $T_{(x, f(x))} S$ 
and 
$(\id_{TX}+(F^S)^\sharp)_* \left(\p/\p y^a \right)$ as an element of $T^\perp_{(x, f(x))}S$. 
Moreover, 
$(\id_{TX}+(F^S)^\sharp)_*|_{W_0}: W_0 \rightarrow T_{(x, f(x))} S$ 
and 
$(\id_{TX}+(F^S)^\sharp)_*|_{V_0}: V_0 \rightarrow T^\perp_{(x, f(x))} S$ 
are isomorphisms, 
where $W_0$ and $V_0$ are subspaces of $\R^8$ 
spanned by $\p/\p x^0, \cdots, \p/\p x^3$ and $\p/\p y^4, \cdots, \p/\p y^7$, 
respectively. 
Since $F^B_\n$ is a linear combination of $dx^{i j}$'s, 
we see that 
\[((\id_{TX}+(F^S)^\sharp)^{-1})^* (-\i F_\n^B) \in \Lambda^2 T^*_{(x, f(x))} S\]
in the sense of Lemma \ref{lem:Cayley ASD}. 
Then, by Lemma \ref{lem:Cayley ASD}, 
\eqref{eq:FM Spin7 conn 2} holds if and only if 
$((\id_{TX}+(F^S)^\sharp)^{-1})^* (-\i F_\n^B) \in \Lambda^2 T^*_{(x, f(x))} S$ 
is anti-self-dual with respect to the induced metric 
and the orientation which makes $S$ Cayley. 

Since the identification between $B$ and $S$ is given by 
$\k: B \ni x \mapsto (x,f(x)) \in S$ 
and $(d \k)_x = (\id_{TX}+(F^S)^\sharp)_*|_{W_0}$, 
where we identify $T_x B$ with $W_0$, 
we obtain the desired statement. 
\end{proof}
%%%%%%%%%%%%%%%%%%%%%%%%%%%%%%%%%%%%%%%%%%%%%%%%%%%%%%%%%
%%%%%%%%%%%%%%%%%%%%%%%%%%%%%%%%%%%%%%%%%%%%%%%%%%%%%%%%%
%%%%%%%%%%%%%%%%%%%%%%%%%%%%%%%%%%%%%%%%%%%%%%%%%%%%%%%%%
%%%%%%%%%%%%%%%%%%%%%%%%%%%%%%%%%%%%%%%%%%%%%%%%%%%%%%%%%
%%%%%%%%%%%%%%%%%%%%%%%%%%%%%%%%%%%%%%%%%%%%%%%%%%%%%%%%%
%%%%%%%%%%%%%%%%%%%%%%%%%%%%%%%%%%%%%%%%%%%%%%%%%%%%%%%%%
\section{The real Fourier--Mukai transform for associative $T^3$-fibrations}
\label{sec:dDT G2 2}
In this section, 
we compute the real Fourier--Mukai transform 
of coassociative cycles using Theorems \ref{thm:FM Spin7} and \ref{thm:FM Spin7 conn}. 
It turns out that 
the real Fourier--Mukai transform of an associative cycle coincides with that of a coassociative cycle as stated in \cite{LL}.

Let $B\subset\mathbb{R}^4$ be an open set with coordinates $(y^{4},y^{5},y^{6},y^{7})$
and 
$f=(f^{1},f^{2},f^{3}):B\to T^{3}$ be a smooth function with values in $T^{3}$, 
where we use coordinates $(x^{1},x^{2},x^{3})$ for $T^{3}$. 
Put
\[S:=\{\,(y,f(y))\mid y \in B\,\}\]
the graph of $f$, a 4-dimensional submanifold in $X:=B\times T^{3}$. 
The manifold $X$ admits a $G_2$-structure $\varphi$ 
with its Hodge dual $* \varphi = *_7 \varphi$ as in \eqref{eq:FM asso} and \eqref{eq:FM coasso}. 
Let 
\[
\n^B = d + \i \sum_{a=4}^7 A^a dy^a
\]
be a Hermitian connection of a trivial complex line bundle $B \times \C \to B$, 
where 
$A^j : B \to \R$ is a smooth function. 

Next, we consider the mirror side. 
The real Fourier--Mukai transform of 
$(S, \n^B)$ is the connection on $X^{*}(\cong X)$ defined by 
\[
\nabla:=d +\i \sum_{a=4}^7 A^{a} dy^{a} + \i \sum_{j=1}^3 f^{j} dx^j. 
\]
Then, its curvature 2-form $F_\nabla$ is given by $F_\nabla = F_\n^B + F_\n^S$, where 
\begin{equation} \label{eq:FM curv 2}
F_\n^B = \i  \sum_{a,b=4}^7 \frac{\partial A^{b}}{\partial y^{a}} dy^a \wedge dy^b, \quad
F_\n^S = \i \sum_{j=1}^3 \sum_{a=4}^7 \frac{\partial f^{j}}{\partial y^{a}} dy^a \wedge dx^j. 
\end{equation}

We first describe the condition for $S$ to be coassociative in terms of $F_\n^S$.

\begin{proposition} \label{prop:FM G2 2}
The following conditions are equivalent. 
\begin{enumerate}
\item
The graph $S$ is a coassociative submanifold with an appropriate orientation. 
\item
$(F_\n^S)^3/6 + F_\n^S \wedge \ast\varphi=0$. 
\item
$(F_\n^S)^3/6 + F_\n^S \wedge \ast\varphi=0
\quad\mbox{and}\quad \varphi\wedge \ast (F_\n^S)^{2}=0.$
\end{enumerate}
\end{proposition}
Thus, we obtain the same equations as in Proposition \ref{prop:FM G2}. 

\begin{proof}
Since (3) obviously implies (2) and the converse holds by \cite[Remark 3.3]{KY}, 
(2) and (3) are equivalent. 
We show the equivalence of (1) and (3). 
Fixing $* \in S^1$, 
we have an embedding 
\[
\iota: B \times T^3 \cong B \times \{\, \ast \,\} \times T^3 \hookrightarrow B \times T^4.
\]
Let $(x^0, x^1, x^2, x^3)$ be coordinates for $T^4$. 
We canonically identify $F_\nabla^S$ on $B \times (T^3)^*$ 
with a 2-form on $B \times (T^4)^*$ such that $ i(\p/\p x^0) F_\nabla^S =0$. 

The manifold $B \times T^4$ admits a ${\rm Spin}(7)$-structure $\Phi$ given by 
\[
\Phi = dx^0 \wedge \varphi + *_7 \varphi, 
\]
and the graph $S$ is coassociative if and only if $\iota(S)$ is Cayley. 
Then, Theorem \ref{thm:FM Spin7} implies that 
$S$ is a coassociative submanifold with an appropriate orientation if and only if 
\[
\pi^6_7 \left( *_8 F_\nabla^S + (F_\nabla^S)^3/6 \right) =0
\quad\mbox{and}\quad  \pi^4_7  \left((F_\nabla^S)^{2} \right)=0.
\]
We describe these equations in terms of the $G_2$-structure $\varphi$ on $B \times T^3$. 
By \eqref{eq:Spin7 proj 3} and \eqref{eq:Hodge 78}, we have 
\[
\begin{aligned}
4 \pi^6_7 \left( *_8 F_\nabla^S \right) 
&=
*_8 F_\nabla^S + \Phi \wedge F_\nabla^S \\
&=
dx^0 \wedge *_7 F_\nabla^S + (dx^0 \wedge \varphi + *_7 \varphi) \wedge F_\nabla^S \\
&=
dx^0 \wedge (*_7 F_\nabla^S + \varphi \wedge F_\nabla^S ) + *_7 \varphi \wedge F_\nabla^S, \\
4 \pi^6_7 \left( (F_\nabla^S)^3 \right) 
&=
(F_\nabla^S)^3 + \Phi \wedge *_8 (F_\nabla^S)^3 \\
&=
(F_\nabla^S)^3 + (dx^0 \wedge \varphi + *_7 \varphi) \wedge dx^0 \wedge *_7 (F_\nabla^S)^3 \\
&=
dx^0 \wedge *_7 F_\nabla^3 \wedge *_7 \varphi + (F_\nabla^S)^3. 
\end{aligned}
\]
Hence, $\pi^6_7 \left( *_8 F_\nabla^S + (F_\nabla^S)^3/6 \right) =0$ is equivalent to 
\[
*_7 F_\nabla^S + \varphi \wedge F_\nabla^S + \frac{1}{6} *_7 (F_\nabla^S)^3 \wedge *_7 \varphi = 0, 
\quad 
*_7 \varphi \wedge F_\nabla^S + \frac{1}{6} (F_\nabla^S)^3 =0. 
\]
Since these two equations are equivalent by \cite[Lemma 3.2]{KY}, 
$\pi^6_7 \left( *_8 F_\nabla^S + (F_\nabla^S)^3/6 \right) =0$ is equivalent to 
$*_7 \varphi \wedge F_\nabla^S + (F_\nabla^S)^3/6 =0$. 

Next, we consider $\pi^4_7  \left((F_\nabla^S)^{2} \right)=0$. 
By Lemma \ref{lem:lambdas}, $\pi^4_7  \left((F_\nabla^S)^2 \right)=0$ if and only if
\[
\la dx^0 \wedge i (\alpha^\sharp) \ast_7 \varphi - \alpha \wedge \varphi, (F_\nabla^S)^2 \ra =0
\]
for any $\alpha \in \Om^1(B \times T^3)$. 
Since $ i(\p/\p x^0) F_\nabla^S =0$, this is equivalent to 
\begin{align*}
0 
= \la \alpha \wedge \varphi, (F_\nabla^S)^2 \ra 
= *_7 (\alpha \wedge \varphi \wedge *_7 (F_\nabla^S)^2)
= \la \alpha, *_7 (\varphi \wedge *_7 (F_\nabla^S)^2) \ra.
\end{align*}
Hence, the proof is completed by \cite[Remark 3.3]{KY}. 
\end{proof}

Similarly, we obtain the following Proposition \ref{prop:FM G2 conn 2} from Theorem \ref{thm:FM Spin7 conn}. 

\begin{proposition} \label{prop:FM G2 conn 2}
The following conditions are equivalent. 
\begin{enumerate}
\item
The graph $S$ is a coassociative submanifold with an appropriate orientation 
and if we identify $-\i F_\n^B \in \Om^2(B)$ with a 2-form on $S$, 
it is anti-self-dual with respect to the induced metric 
and the orientation which makes $S$ coassociative. 
\item
$F_\n^3/6 + F_\n \wedge \ast\varphi=0$. 
\item
$F_\n^3/6 + F_\n \wedge \ast\varphi=0
\quad\mbox{and}\quad \varphi\wedge \ast F_\n^{2}=0.$
\end{enumerate}
\end{proposition}

Since the lemma corresponding to Lemma \ref{lem:Cayley ASD} would be interesting in itself, 
we write it down here.

\begin{lemma} \label{lem:coasso ASD}
Let $U \subset \R^7$ be a coassociative subspace, a subspace of $\R^7$ which is a coassociative submanifold. 
Denote by $U^\perp$ the orthogonal complement of $U$. 
We identify $\Lambda^k U^*$ with the subspace of $\Lambda^k (\R^7)^*$ by 
\[
\Lambda^k U^* = \{\, \alpha \in \Lambda^k (\R^7)^* \mid i(v) \alpha =0 
\mbox{ for any } v \in U^\perp \,\}. 
\]
Then 
$\alpha \in \Lambda^2 U^*$ is anti-self-dual with respect to the induced metric 
if and only if $\alpha \wedge * \varphi =0$. 
\end{lemma}

\begin{proof}
Since $U$ is coassociative, there is an orthonormal basis $\{\, e_i \,\}_{i=1}^7$ with its dual $\{\, e^i \,\}_{i=1}^7$ 
such that 
$U$ is spanned by $e_4, \cdots, e_7$, 
which is positively oriented, 
$U^\perp$ is spanned by $e_1, \cdots, e_3$  
and \eqref{varphi} holds. 
Setting $\om_1=e^{45} + e^{67}, \om_2=e^{46} - e^{57}$ and $\om_3=-(e^{47} + e^{56})$, 
we have $* \varphi = e^{4567} + \sum_{k \in \Z/3} e^{k, k+1} \wedge \om_{k+2}$. 
Then, it follows that 
\[
\alpha \wedge * \varphi 
= \sum_{k \in \Z/3} e^{k, k+1} \wedge \alpha \wedge \om_{k+2} 
= \sum_{k \in \Z/3} e^{k, k+1} \wedge \la \alpha, \om_{k+2} \ra e^{4567}. 
\]
Since $\{ e_4, \cdots, e_7 \}$ is positively oriented, 
$\{\, \om_1, \om_2, \om_3 \,\}$ is a basis of the space of self-dual 2-forms on $U$. 
Hence, the proof is completed. 
\end{proof}

By this lemma and results in \cite{KY}, we can also prove Proposition \ref{prop:FM G2 conn 2} 
without using Theorem \ref{thm:FM Spin7 conn}.

%%%%%%%%%%%%%%%%%%%%%%%%%%%%%%%%%%%%%%%%%%%%%%%%%%%%%%%%%
%%%%%%%%%%%%%%%%%%%%%%%%%%%%%%%%%%%%%%%%%%%%%%%%%%%%%%%%%
%%%%%%%%%%%%%%%%%%%%%%%%%%%%%%%%%%%%%%%%%%%%%%%%%%%%%%%%%
%%%%%%%%%%%%%%%%%%%%%%%%%%%%%%%%%%%%%%%%%%%%%%%%%%%%%%%%%
%%%%%%%%%%%%%%%%%%%%%%%%%%%%%%%%%%%%%%%%%%%%%%%%%%%%%%%%%
%%%%%%%%%%%%%%%%%%%%%%%%%%%%%%%%%%%%%%%%%%%%%%%%%%%%%%%%%
%%%%%%%%%%%%%%%%%%%%%%%%%%%%%%%%%%%%%%%%%%%%%%%%%%%%%%%%%
%%%%%%%%%%%%%%%%%%%%%%%%%%%%%%%%%%%%%%%%%%%%%%%%%%%%%%%%%
%%%%%%%%%%%%%%%%%%%%%%%%%%%%%%%%%%%%%%%%%%%%%%%%%%%%%%%%%
%%%%%%%%%%%%%%%%%%%%%%%%%%%%%%%%%%%%%%%%%%%%%%%%%%%%%%%%%
%%%%%%%%%%%%%%%%%%%%%%%%%%%%%%%%%%%%%%%%%%%%%%%%%%%%%%%%%
\section{Compatibilities with other connections} \label{sec:suggestSpin7dDT}

In this section, we post some evidences showing that Definition \ref{def:Spin7dDTMain} we suggest 
is compatible with deformed Donaldson--Thomas (dDT) connections for a $G_2$-manifold  
and deformed Hermitian Yang--Mills (dHYM) connections of a Calabi-Yau 4-manifold.

Use the notation (and identities) of Subsection \ref{sec:Spin7 geometry}. 
Let 
$X^8$ be a compact connected 8-manifold with a ${\rm Spin}(7)$-structure $\Phi$ 
and $L \to X$ be a smooth complex line bundle with a Hermitian metric $h$.
Set 
\[
\begin{aligned}
\mathcal{A}_{0}=&\{\,\nabla \mid \mbox{a Hermitian connection of }(L,h) \,\} 
= \nabla + \i \Om^1 \cdot \id_L
\end{aligned}
 \]
for a fixed connection $\nabla \in \Aa_{0}$. 
We regard the curvature 2-form $F_\n$ of $\n$ as a $\i \R$-valued closed 2-form on $X$. 

Define maps 
$\Ff^1_{{\rm Spin}(7)}:\Aa_{0} \rightarrow \i \Om^{2}_7$ and 
$\Ff^2_{{\rm Spin}(7)}:\Aa_{0} \rightarrow \Om^{4}_7$ by 
\begin{equation*}
\begin{aligned}
\Ff^{1}_{{\rm Spin}(7)}(\nabla)
&= 
\pi^2_{7} \left( F_\nabla + \frac{1}{6} * F_\nabla^3 \right) 
= 
\frac{1}{4}
\left(
F_\nabla + \frac{1}{6} * F_\nabla^3 + 
* \left( \left( F_\nabla + \frac{1}{6} * F_\nabla^3  \right) \wedge \Phi \right) 
\right), \\
\Ff^{2}_{{\rm Spin}(7)}(\nabla)
&=
\pi^{4}_{7}(F_{\nabla}^2). 
\end{aligned}
\end{equation*}
Then, 
a Hermitian connection $\nabla$ of $(L,h)$ satisfying 
$$
\Ff^{1}_{{\rm Spin}(7)}(\nabla)=0 \quad \mbox{and} \quad \Ff^{2}_{{\rm Spin}(7)}(\nabla)=0 
$$
is a $\Sp$-dDT connection defined in Definition \ref{def:Spin7dDTMain}.

\begin{lemma} \label{lem:Spin7 G2}
Let $(Y^7,\varphi,g)$ be a $G_2$-manifold  
with the Hodge dual $*_7 \varphi \in \Om^4$. 
Then, $X^8=S^1 \times Y^7$ is a $\Sp$-manifold. 
Let $L \rightarrow Y$ be a smooth complex line bundle with a Hermitian metric $h$. 
Identify a connection $\nabla$ on $Y^7$ 
with that on $X^8$ by the pullback. 
Then, the following are equivalent. 
\begin{enumerate}
\item
$\n$ is a dDT connection in the sense of $G_2$, that is, 
$\Ff_{G_2}(\n)= *_7 \varphi \wedge F_\nabla + F_\nabla^3/6 = 0$. 
\item
$\Ff^{1}_{{\rm Spin}(7)}(\nabla)=0$. 
\item
$\Ff^{1}_{{\rm Spin}(7)}(\nabla)=\Ff^{2}_{{\rm Spin}(7)}(\nabla)=0$. 
\end{enumerate}
\end{lemma}

\begin{proof}
Recall that the induced ${\rm Spin}(7)$-structure on $X^8$ is given by 
\[
\Phi= dx \wedge \varphi + *_7 \varphi,  
\]
where $x$ is a coordinate of $S^1$ and $*_7$ is the Hodge star on $Y^7$. 
By \eqref{eq:Hodge 78}, we have  
\[
* F_\nabla = dx \wedge *_7 F_\nabla, \quad * F_\nabla^3 = dx \wedge (*_7 F_\nabla^3), 
\]
where $*=*_8$  is the Hodge star on $X^8$. 
Then, we have 
\[
\begin{aligned}
4 * \Ff^{1}_{{\rm Spin}(7)}(\nabla) 
=&
dx \wedge *_7 F_\nabla + \frac{1}{6} F_\nabla^3 
+ F_\nabla \wedge (dx \wedge \varphi + *_7 \varphi) 
+ \frac{1}{6} dx \wedge (*_7 F_\nabla^3) \wedge *_7 \varphi \\
=&
dx \wedge \left( *_7 F_\nabla + \varphi \wedge F_\nabla 
+ \frac{1}{6} (*_7 F_\nabla^3) \wedge *_7 \varphi \right) 
+*_7 \varphi \wedge F_\nabla + \frac{1}{6} F_\nabla^3. 
\end{aligned}
\]
Thus, we see that (1) and (2) are equivalent by \cite[Lemma 3.2]{KY}. 

The equivalence of (2) and (3) follows from \cite[Proposition 3.3]{KYSpin7} since $F_\n^4=0$. 
This equivalence can also be proved by \cite[Remark 3.3]{KY}. 
\end{proof}

\begin{lemma}\label{lem:Spin7 CY4}
Let  $(X^8,J,g,\om,\Om)$ be a Calabi--Yau 4-manifold and 
$L \rightarrow X$ be a complex line bundle with a Hermitian metric $h$. 
Equip $X$ with a $\Sp$-structure $\Phi$ given by 
\[
\Phi=\frac{1}{2} \om^2 + \mathop{\mathrm{Re}} \Om. 
\]
Suppose that $\nabla$ is a Hermitian connection such that 
the $(0,2)$-part $F^{0,2}_{\nabla}$ of $F_{\nabla}$ vanishes. 
Then, we have $\Ff^{2}_{{\rm Spin}(7)} (\nabla)=0$. 
Moreover, 
$\n$ is a dHYM connection with phase $1$ on $X^8$, 
that is, 
$\mathop{\mathrm{Im}}\left( \omega + F_\nabla \right)^4=0$, 
if and only if 
$\n$ is a $\Sp$-dDT connection, that is, 
$\Ff^{1}_{{\rm Spin}(7)} (\nabla)=0$. 
\end{lemma}
\begin{proof}
By \cite[Proposition 2]{Munoz}, $\Lambda^4_7 T^* X$ is contained in 
the space of (3,1), (1,3), (4,0) and (0,4)-forms. 
Since $F_\n^{0,2} = 0$, $F_\n^2$ is a real (2,2)-form, 
which implies that $\Ff^{2}_{{\rm Spin}(7)} (\nabla) = \pi^4_7 (F_\n^2)=0$.

Next, we show the second statement. 
By \cite[Proposition 2]{Munoz}, we have 
$\Lambda^2_7 T^* X = \R \om \oplus A_{+}$, 
where $A_{+}$ is a subspace of $\Lambda^{2,0} T^* X \oplus \Lambda^{0,2} T^* X$. 
Then, we have 
$$
\la A_{+}, \Ff^{1}_{{\rm Spin}(7)}(\nabla) \ra
=\la A_{+}, F_\nabla + * F_\nabla^3/6 \ra =0
$$ 
since $F_\n$ is a (1,1)-form. 
Thus, 
$\Ff^{1}_{{\rm Spin}(7)}(\nabla)=0$ if and only if 
$\la \om, F_\nabla + * F_\nabla^3/6 \ra=0$. 
Since $* \om = \om^3/6$, we have 
\begin{align*}
\left \la \om, F_\nabla + \frac{1}{6} * F_\nabla^3 \right \ra \vol
&= * \om \wedge F_\n + \frac{1}{6} \om \wedge F_\n^3 \\
&=
\frac{1}{6} \left( \om^3 \wedge F_\nabla + \om \wedge F_\nabla^3 \right) 
=
\frac{\i}{24} \Im (\om + F_\nabla)^4. 
\end{align*}
Hence, the proof is completed. 
\end{proof}

\begin{remark}
Note that dHYM connections do not depend on the holomorphic volume form $\Om$.  
Then, since 
$\left(X^8,J,g,\om, e^{-\i \theta} \Om \right)$ is again a Calabi--Yau manifold for $\theta \in \R$, 
Lemma \ref{lem:Spin7 CY4} implies that for a Hermitian connection $\n$ with $F_\n^{0,2}=0$, 
$\n$ is a dHYM connection with phase 1 if and only if 
$\n$ is a $\Sp$-dDT connection with respect to $\Phi_\theta = \om^2/2 + \mathop{\mathrm{Re}} (e^{- \i \theta} \Om)$. 
\end{remark}

%%%%%%%%%%%%%%%%%%%%%%%%%%%%%%%%%%%%%%%%%%%%%%%%%%%%%%%%%
%%%%%%%%%%%%%%%%%%%%%%%%%%%%%%%%%%%%%%%%%%%%%%%%%%%%%%%%%
%%%%%%%%%%%%%%%%%%%%%%%%%%%%%%%%%%%%%%%%%%%%%%%%%%%%%%%%%
%%%%%%%%%%%%%%%%%%%%%%%%%%%%%%%%%%%%%%%%%%%%%%%%%%%%%%%%%
%%%%%%%%%%%%%%%%%%%%%%%%%%%%%%%%%%%%%%%%%%%%%%%%%%%%%%%%%
%%%%%%%%%%%%%%%%%%%%%%%%%%%%%%%%%%%%%%%%%%%%%%%%%%%%%%%%%
\appendix
\section{Notation} \label{app:notation}
We summarize the notation used in this paper. 
We use the following for a manifold $X$ with a $G_2$- or $\Sp$-structure. 
Denote by $g$ the associated Riemannian metric.

\vspace{0.5cm}
\begin{center}
\begin{tabular}{|lc|l|}
\hline
Notation                        & & Meaning \\ \hline \hline
 $i(\,\cdot\,)$                &  & The interior product \\
$\Gamma(X, E)$            & &  The space of all smooth sections of  a vector bundle $E \rightarrow X$\\
$\Om^k$                       & & $\Om^k = \Om^k (X) =  \Gamma(X, \Lambda^k T^*X)$ \\
$v^{\flat} \in T^* X$       & & $v^{\flat} = g(v, \,\cdot\,)$ for $v \in TX$ \\ 
$\alpha^{\sharp} \in TX$ & & $\alpha = g(\alpha^{\sharp}, \,\cdot\,)$ for $\alpha \in T^* X$ \\
$\vol$                           & & The volume form induced from $g$ \\
$\Lambda^k_\l T^*X$      & & \hspace{-1ex}\begin{tabular}{l}The subspace of $\Lambda^k T^*X$ corresponding to  
                                   an $\l$-dimensional \\ irreducible subrepresentation as in Subsection \ref{sec:Spin7 geometry} \end{tabular} \\
$\Om^k_\l$                    & &  $\Om^k_\l = \Gamma (X, \Lambda^k_\l T^*X)$ \\
$\pi^k_\l$                       & & The projection $\Lambda^k T^*X \rightarrow \Lambda^k_\l T^*X$ or 
                                           $\Om^k \rightarrow \Om^k_\l$ \\
\hline
\end{tabular}
\end{center}
%%%%%%%%%%%%%%%%%%%%%%%%%%%%%%%%%%%%%%%%%%%%%%%%%%%%%%%%%
%%%%%%%%%%%%%%%%%%%%%%%%%%%%%%%%%%%%%%%%%%%%%%%%%%%%%%%%%
%%%%%%%%%%%%%%%%%%%%%%%%%%%%%%%%%%%%%%%%%%%%%%%%%%%%%%%%%
%%%%%%%%%%%%%%%%%%%%%%%%%%%%%%%%%%%%%%%%%%%%%%%%%%%%%%%%%
%%%%%%%%%%%%%%%%%%%%%%%%%%%%%%%%%%%%%%%%%%%%%%%%%%%%%%%%%
%%%%%%%%%%%%%%%%%%%%%%%%%%%%%%%%%%%%%%%%%%%%%%%%%%%%%%%%%


\begin{thebibliography}{99}
\bibitem{Bryant}
R. Bryant. Some remarks on $G_2$-structures. 
{\it Proceedings of G\"okova Geometry-Topology Conference 2005}, 75--109, {\it G\"okova Geometry/Topology Conference (GGT), G\"okova}, 2006.

\bibitem{DK}
S. K. Donaldson, P. B. Kronheimer. The geometry of four-manifolds. 
Oxford Mathematical Monographs. Oxford Science Publications. 
{\it The Clarendon Press, Oxford University Press, New York}, 1990. {\rm x}+440 pp. 

\bibitem{HL}
R. Harvey and H. B. Lawson. Calibrated geometries, Acta Math. 148 (1982), 47--157.

\bibitem{Kar}
S. Karigiannis. Deformations of $G_2$ and ${\rm Spin}(7)$ structures. {\it Canad. J. Math.} {\bf 57} (2005), no. 5, 1012--1055.


\bibitem{KLS}
K. Kawai, H. V. L\^e and L. Schwachh\"ofer. 
The Fr\"olicher-Nijenhuis  bracket and the geometry of $G_2$- and ${\rm Spin}(7)$-manifolds. 
{\it Ann. Mat. Pura Appl. (4)} {\bf 197} (2018), no. 2, 411--432. 

\bibitem{KY}
K. Kawai and H. Yamamoto. 
Deformation theory of deformed Hermitian Yang-Mills connections 
and deformed Donaldson-Thomas connections. 
arXiv:2004.00532. 

\bibitem{KYSpin7}
K. Kawai and H. Yamamoto. 
Deformation theory of deformed Donaldson--Thomas connections for $\Sp$-manifolds. 
to appear in J. Geom. Anal., arXiv:2101.03986. 

\bibitem{KYvol}
K. Kawai and H. Yamamoto. 
Mirror of volume functionals on manifolds with special holonomy. 
arXiv:2103.13863. 

\bibitem{LL}
J.-H. Lee and N. C. Leung. 
Geometric structures on $G_2$ and ${\rm Spin}(7)$-manifolds. 
{\it Adv. Theor. Math. Phys.} {\bf 13} (2009), no. 1, 1--31. 

\bibitem{LYZ}
N.-C. Leung, S.-T. Yau and E. Zaslow. 
From special Lagrangian to Hermitian--Yang--Mills via Fourier--Mukai transform. 
{\it Adv. Theor. Math. Phys}. {\bf 4} (2000), no. 6, 1319--1341. 

\bibitem{McLean}
R.C. McLean. 
Deformations of calibrated submanifolds. 
{\it Comm. Anal. Geom}. {\bf 6} (1998), no. 4, 705--747. 

\bibitem{Munoz}
V. Mu\~noz. $\rm Spin(7)$-instantons, stable bundles and the Bogomolov inequality for complex 4-tori. 
{\it J. Math. Pures Appl. (9)} {\bf 102} (2014), no. 1, 124--152. 
\end{thebibliography}
\end{document}